\documentclass[reqno]{amsart}
\usepackage{amssymb,latexsym,amscd,graphicx}
\usepackage{color}
\newtheorem{lem}{Lemma}[section]
\newtheorem{cor}{Corollary}[section]
\newtheorem{thm}{Theorem}[section]

\theoremstyle{definition}

\theoremstyle{remark}

\definecolor{orange}{rgb}{1,0.5,0}
\usepackage{graphicx}
\graphicspath{{converted_graphics/}}

\begin{document}

\title{ Automorhpic integrals with rational period functions and arithmetical identities}          %Enter your title between curly braces
\author{Tewlede G/Egziabher, Hunduma Legesse Geleta, Abdul Hassen }        % Enter your name between curly braces
\address{Department of Mathematics,Ambo University, Ambo, ethiopia
Addis Ababa University, Addis Ababa, ethiopia
 Rowan University, Glassboro, NJ 08028.}
\email{tewelde.geber@ambou.edu.et,  hassen@rowan.edu }
 %\thanks{}

\date{\today}          % Enter your date or \today between curly braces
\maketitle

 \begin{abstract}   In 1961, Chandrasekharan and Narasimhan  showed that for a large class of Dirichlet series the functional equation and two types of arithmetical identities are equivalent.  In 1992, Hawkins and  Knopp proved a Hecke correspondence theorem for modular integrals with rational period function on theta group. Analogous to Chandrasekharan and Narasimhan, in 2015 Sister Ann M. Heath  has shown that the functional equation in Hawkins and Knopp context and two type of arithmetical identities are equivalent. She considered the functional equation and showed its equivalence to two arithmetical identities associated with entire modular cusp integrals involving rational period functions for the full modular group. In this paper we extend the results of Sister Ann M. Heath to entire automorphic integrals involving rational period functions on discrete Hecke group. 
 \end{abstract}

\maketitle
\def\theequation{\thesection.\arabic{equation}}
\section{INTRODUCTION}
\setcounter{equation}{0}

 Let $ \lbrace \lambda_{n}\rbrace$ and $\lbrace \mu_{n} \rbrace$ be two strictly increasing sequences of positive real numbers  diverging to  $\infty,$ as $n\rightarrow \infty$ and let $ \lbrace a_{n} \rbrace$ and  $ \lbrace b_{n} \rbrace$ be two sequences of complex numbers not identically zero. Consider the Dirichlet series $\varphi $ and $\psi$ defined by  
 $$\varphi(s)=\sum_{n=1}^{\infty} \frac{a_{n}}{\lambda_{n}^{s}}\ \ \ \ \mbox{ and } \ \ \ \ \psi(s)=\sum_{n=1}^{\infty} \frac{b_{n}}{\mu_{n}^{s}} $$ with finite abscissas of absolute convergence $\sigma_{a}$ and $\sigma_{b},$ respectively. Suppose that  $\varphi$ and $\psi $ satisfy the functional equation
 \begin{equation}\label{Bocn}
 \Gamma(s)\varphi(s)=\Gamma(\delta-s)\psi(\delta-s),
 \end{equation} 
 where $\delta >0.$\ 
 Chandrasekharan and Narasimhan  \cite{CN} showed that the functional equation (\ref{Bocn}) is equivalent to the following arithmetical  identities (1.2) and (1.3).\\
 
 \begin{equation} \label{b} 
 \frac{1}{\Gamma(\rho +1)}{\sum_{\lambda_n \leq x}}^\prime
 a_n(x - \lambda_n)^\rho 
 = \left( \frac{1}{2\pi}\right)^\rho \sum_{n=1}^\infty
 b_n \left( \frac{x}{\mu_n} \right)^{\frac{\delta +
 		\rho}{2}}  J _{\delta + \rho}  \{4 \pi \sqrt{\mu_n x} \} + Q_\rho (x) \, ,
 \end{equation}
where $x > 0$, $\rho \geq 2 \beta -\delta -\frac{1}{2},$~~ \  $J_\nu(z)$ denotes the usual Bessel function of the first kind of order $\nu,$ 
$$ Q_\rho(x) = \frac{1}{2\pi i} \oint_c \frac{ \Phi(s)(2\pi)^s x^{s+ \rho}}{\Gamma(s + \rho +
 	1)}ds,~~~~\qquad \sum_{n=1}^\infty \mid b_{n}\mid \mu_{n}^{-\beta} < \infty.$$
   
 \begin{equation} \label{c} 
 \left( -\frac{1}{ s}\frac{d}{ds}\right)^\rho \left[ \frac{1}{s} 
 \sum^\infty_{n=1} a_n e^{-s\sqrt{\lambda_n}} \right] 
 = 2^{3\delta+\rho}
 \Gamma(\delta+\rho+\frac{1}{2})\pi^{\delta-\frac{1}{2}}
 \sum^\infty_{n=1} \frac{b_n}{ (s^2 + 16\pi^2\mu_n)^{\delta+\rho+\frac{1}{2}}}
 + R_\rho(s).
 \end{equation}

 where   Re $s > 0, $ $\rho $  is non-negative integer satisfying 
 $\rho \geq  \beta -\delta -\frac{1}{2}$~ ~and
 
 \[ R_\rho(s) = \frac{1}{2\pi i} \oint_c \frac{\Phi(z) (2\pi)^z \Gamma(2z +
 	2\rho + 1) 2^{-\rho}} { \Gamma(z + \rho + 1)} s^{-2z-2\rho-1} dz. 
 \]
 Note that  if $\beta > 0$, then identity   (\ref{c}) holds for $\rho$ satisfying  $\rho \geq \beta -\delta -\frac{1}{2},$   $\rho \in\mathbb{Z}_{\geq 0} $.\\
 
 In \cite{HK1992},  Hawkins and Knopp proved a Hecke correspondence theorem for modular integrals with rational period functions on $\Gamma_{\theta}$, (generated by $Sz=z+2$ and $Tz=-1/z$), a subgroup of the full modular group $\Gamma(1)$.  In their work,  the   functional equation takes the form
 \begin{equation}\label{HKCON}
 \Phi(2k-s)-i^{2k}\Phi(s)=R_{k}(s),
 \end{equation}
 where $$\Phi(s)=\left(\frac{2\pi}{\lambda_{n}}\right)^{-s}\Gamma(s)\sum_{n=1}^{\infty}a_{n}n^{-s}$$
 is associated with a modular relation involving rational period function $q(z)$   of the form
 
 \begin{equation}\label{ModRel} F(z+\lambda)=F(z)\ \ \ \  \mbox{and} \ \ \ \ z^{-2k}F\left( \frac{-1}{z}\right)=F(z)+q(z), \end{equation} where $
 \lambda=\lambda_{n}=2\cos\left( \frac{\pi}{n}\right),$  with    $3\leq n\in\mathbb{N}\bigcup\lbrace \infty\rbrace $ and $2k\in\mathbb{Z}.$ \\
 
 Analogous to Chandrasekharan and Narasimhan,  Sister Ann M. Heath  \cite{SISANN}  showed that the functional equation in the Hawkins and Knopp context (\ref{HKCON}) and the arithmetical identities are equivalent. To prove her results, she used the fact that the correspondence theorem between the functional equational   (\ref{HKCON}) and  the associated entire  integrals form with rational period function for the full modular group $\Gamma(1).$  We \cite{THA} also used  techniques of Chandrasekharan and Narasimhan to prove results analogous to those of Sister Ann M. Heath and established equivalence of two arithmetical identities with functional equation associated with automorphic integrals involving log-polynomial-period functions on the discrete Hecke group. \\
    
In this paper we use the techniques of Chandrasekharan and Narasimhan \cite{CN} and extend the results of Sister Ann M. Heath \cite{SISANN} to entire automorphic integrals involving rational period functions on discrete Hecke group $G(\lambda)$. \\

This paper is organized as follows: In section two we review some results concerning, Heck groups, automorphic integrals with rational period function and present some preliminary results. In section three we present our main results with their proofs.

 \maketitle
\def\theequation{\thesection.\arabic{equation}}
\section{Preliminaries}
\setcounter{equation}{0}
In this section we review some terms and results that are useful in the coming sections.\\

 Recall that    the Hecke group    $G(\lambda)$,  where $\lambda\in\mathbb{R^{+}}$,       is defined as the subgroup of    $SL_{2}(\mathbb{R})$   given by $$G(\lambda)=\left\langle 
\begin{pmatrix} 
1 &\lambda \\ 0& 1
\end{pmatrix},
\begin{pmatrix} 0& 1 \\ -1& 0 
\end{pmatrix}\right\rangle . $$
Equivalently  $G(\lambda)$ is generated by the  linear fractional transformations $S(z)=z+\lambda$    and    $T(z)=-\frac{1}{z}$. The element of $G(\lambda)$ act on the Riemann sphere as linear fractional transformation, that is    $Mz=\frac{az+b}{cz+d}$   for    $ M=\begin{pmatrix}a & b\\ c & d \end{pmatrix}$ $ \in G(\lambda),$   and    $z\in\mathbb{C}\bigcup\lbrace \infty\rbrace,$   thus    $M$   and    $-M$   can be identified with the same linear fractional transformations. Hecke \cite{EHECKE}   showed the group    $G(\lambda)$    is discrete (operates discontinuously) as a set of linear fractional transformations on the upper half plane    $ \mathcal{H}=\{ z=x + i y: y>0\}$    if and only if either    $\lambda>2$   or $\lambda=\lambda_{p}:=2cos(\frac{\pi}{p}),$   with    $3\leq p\in\mathbb{N}\bigcup\lbrace \infty\rbrace .$ ~For $ \lambda \geq 2 $ and $ \lambda=\lambda_{p} ,$~we have the following relations respectively, $$T^{2}= -I, ~~~T^{2}=(S_{\lambda_{p}}T)^{p}=-I.$$ 
It is clear that $G(\lambda_3)=G(1) =\Gamma(1)$    is  the {\it full modular group}, and  $G(\lambda_{\infty}) =\Gamma_{\theta}$  is   the familiar theta group.\\
 
Suppose  $F(z)$   is a meromorphic function in the upper half plane   $\mathcal{H}$  that satisfies (\ref{ModRel}).   Further assume that   $ F$   has Fourier series expansion of the form  
\begin{eqnarray}\label{EN3}
F(z)=\sum^\infty_{m=\nu}a_me^{2\pi imz/\lambda},
\end{eqnarray}  
where $ \Im z =y>y_{o}\geq 0$ and $\nu \in \mathbb{Z}$. The function $F$ is called an {\it automorphic integral of weight $2k$   for the Hecke group   $G(\lambda),$    with rational period function   $\left( RPF\right) ,$    $q(z)$}.    If $q\equiv 0$   then   $F$   is an {\it automorphic form of weight   $2k$   on   $G(\lambda)$} . ~If   $F$   is an automorphic integral and   holomorphic in $ \mathcal{H}$    (that is,    $ \nu \geq 0$) and satisfies  the growth condition  $$ \mid F(z)\mid \leqslant C\left(  \mid z\mid^{\alpha}+ y^{-\beta}\right) ,      \Im(z)=y>0, $$    for some constants    $ C, \alpha ,\beta  >0 ,$   and    $ z \in \mathcal{H}$ ,      one can show that the coefficients   $a_m$   in    $(\ref{EN3})$     satisfy   $$a_m=\mathcal{O}(m^{\beta}),  \ \  m\rightarrow\infty.$$ In this case,     $ F$    is called an \textit{ entire automorphic integral  of weight    $2k$   on   $G(\lambda)$   with rational period function $q$}. \\

For   $ M=\begin{pmatrix}* & *\\ c & d \end{pmatrix}\ \in\ G(\lambda)$    the {\em stroke or slash  operator}is defined by $$F|M:= F|^{M}_{2k}=\left( cz+d\right)^{-2k}F(Mz).$$ 
Thus the second automorphic relation   $(\ref{ModRel})$   can be expressed  as $F|T=F+q.$  In general, for any $M$ in $G(\lambda)$ there is a corresponding   period function $q_M$ such that   $  F|M=F+q_{M}$.   The slash operator satisfies    $F|M_{1}M_{2}=(F|M_{1} )|M_{2}$  for   $ M_{1}, M_{2}\in G(\lambda)$ and hence the family of periodic functions $\{q_M: M\in G(\lambda\}$ are related by 
 \begin{equation}\label{ENN4}
 q_{M_{1}M_{2}}= q_{M_{1}}| M_{2} + q_{M_2},    M_{1},M_{2}\in G(\lambda).
 \end{equation}  
 
     Using the relation $T^{2}=-I$,(\ref{ENN4} imposes a relation on the (RPF)~ $q,$ 
 \begin{equation}\label{ENN5}
   q|T+q=0.
 \end{equation}
 And using the relation  $\left(S_{\lambda_{p}}T \right)^{p}=-I$ for t$\lambda=\lambda_p =  2\cos\left( \frac{\pi}{p}\right) ,   p\in\mathbb{Z}, p\geq 3$, imposes another condition on (RPF)~ $q$,  
 \begin{equation}\label{ENN6}
 q\mid\left(S_{\lambda_{p}}T\right)^{p-1} +q\mid\left(S_{\lambda_{p}}T\right)^{p-2}+\cdots+q\mid\left(S_{\lambda_{p}}T\right)+q=0.
 \end{equation} 
 
 Marvin Knopp   \cite{KP2}    proved that the finite poles of a rational period function on    $\Gamma(1)$    are  only at    $0$   or real quadratic irrationals. He also showed that if    $q$   is a    RPF   of weight   $2k>0$   with poles in   $\mathbb{Q}$,   then for some constants   $\alpha_0,\alpha_1\in\mathbb{C},$
 
 \begin{equation*}
 q(z)=
 \begin{cases}
 \alpha_{0}\left(1-\frac{1}{z^{2k}}\right)&\text{if}  \ \   k>1,\\
 \alpha_{0}\left(1-\frac{1}{z^{2}}\right)+\frac{\alpha_1}{z}&\text{if} \ \    k = 1.
 \end{cases}
 \end{equation*}

 Observe that  if    $F(z)\equiv -\alpha_0$,   then   $\left(F|T\right)(z)=F(z)+q(z)$   implies that   $z^{-2k}F\left(\frac{-1}{z} \right)-F(z)=q(z)$ and hence     $q(z)=\alpha_0\left(1-z^{-2k} \right)$. Thus we consider $q(z)=\alpha_0\left(1-z^{-2k} \right)$~  as the trivial period function of weight    $2k\in\mathbb{R}.$  The following lemma is stated in the work of  Hawkins and  Knopp   \cite{HK1992},  where their underlying group is $\Gamma_{\theta}$ and generalized  to the general Hecke group and multiplier system by Hassen \cite{ABHA}.   
 \begin{lem}\label{LEMA1}
 	Nontrivial rational period function on the Hecke groups satisfying    $(\ref{ENN5})$   and   $(\ref{ENN6})$   exists only if the weight    $2k$   is an integer.
 \end{lem}

 Wendell-Culp-Ressler in   (\cite{WR}, Lemma 3)  showed that the poles of any rational periodic function   $q$    of weight    $2k$,    $k\in\mathbb{Z^+}$   on   $G(\lambda)$   are real numbers. He also proved that the order of a nonzero pole of a     RPF   of  weight    $2k$   on   $G(\lambda)$   is    $[k]$.  With appropriate modifications to fit for the current context of functions on $G(\lambda),$ the work of Hawkins and Knopp \cite{HK1992} can be used to state a special form of ~$RPF$ for the solution of (\ref{ENN5}). This form is given by the following lemma.\\
 
 \begin{lem}\label{LEMA2}
 	For   $r\in\mathbb{Z},   \alpha_{j}\in\mathbb{R}\setminus\lbrace 0\rbrace,   C_{r}, C_{rj}\in\mathbb{C}$   for   $j=1, 2 \cdots, p, ~$    let
 	$f_{r}(z,0)=z^{-r}-(-1)^{r}z^{-2k+r},$~~and\\
 	~~~$f_{r}(z,\alpha_{j})=(z-\alpha_{j})^{-r}-(-1)^{r}\alpha_{j}^{-r}z^{-2k+r}\left(z+\frac{1}{\alpha_{j}} \right)^{-r}.$  Then  
  \begin{equation} \label{RPF1}
 	q(z)=\sum_{k\leq r\leq L}C_{r}f_{r}(z,0)+\sum_{j=1}^{p}\sum_{r=1}^{M_j}C_{rj}f_{r}(z,\alpha_{j})~~and~ satisfies~~ q\mid T+q=0.
 	\end{equation}
 		
 \end{lem}

 \begin{thm}\label{ThM1}
 	
 	Suppose    $F$   is an entire automorphic integral function of weight   $2k,   k\in\mathbb{Z^+}$   for    $G(\lambda)$   with   (RPF  ) $q(z)$, where $q$ has the form described by Lemma \ref{LEMA2}.  Suppose further that   $F$   has a Fourier series expansion of the form 
 	\begin{equation}\label{FEofIF}F(z)=\sum_{m=0}^{\infty}a_{m}e^{2\pi imz/\lambda}, \ \ \ \ \mbox{ with } \ \ \ a_{m}=\mathcal{O}(m^{\beta})\ \ \beta>0, \ m \rightarrow\infty.\end{equation}
 	
 	\begin{equation}\label{Phi1} \varphi(s)=\sum_{m=1}^{\infty}a_{m}m^{-s} \ \ \mbox{ and }\ \  \Phi(s)=\left(\frac{2\pi}{\lambda} \right)^{-s}\Gamma(s)\varphi(s),~~ ~~ s=\sigma+it .\end{equation} 
 	Then   $\Phi(s)$   has a meromorphic continuation to the whole complex plane and can be expressed in the form of
 	$$\Phi(s)=D(s)+D^{0}(s)+E^{0}(s)+E^{H}(s)+E^{B}(s),$$ where
 	\begin{equation}\label{Dse}
 	D(s)=\int_{1}^{\infty}\left(F(iy)-a_{0}\right)\Bigg\{y^{s}+i^{2k}y^{2k-s}\Bigg\}\frac{dy}{y},
 	\end{equation}
 	\begin{equation}\label{Doe}
 	D^{0}(s)=-a_{0}\bigg\{\frac{1}{s}-\frac{i^{2k}}{s-2k}\bigg\},
 	\end{equation}
 	\begin{equation}\label{Eoe}
 	E^{0}=\sum_{k\leq r\leq L} C_{r}(-i)^{r}\Bigg[\frac{1}{r-s} +\frac{i^{2k}}{r-(2k-s)}\Bigg],
 	\end{equation}
 	\begin{align}\label{Ehe}
 	E^{H} (s)& =  -\sum_{j=1}^{p} \sum_{r=1}^{Mj} C_{rj}\frac{(-i)^{r}}{(i\alpha_{j}+1)^{r}}\Bigg\{\frac{1}{s}{_{2}F_{1}}\left[ 1,r;1+s;\frac{1}{(i\alpha_{j}+1)} \right]  \nonumber \\
 	&+ \frac{i^{2k}}{(2k-s)}{_{2}F_{1}}\left[ 1, r; 1+(2k-s); \frac{1}{(i\alpha_{j}+1)}\right] \Bigg\} 	
 	\end{align}
 	\begin{align}\label{Ebe}
 	E^{B}(s)&= i^{2k} \sum_{j=1}^{p} \sum_{r=1}^{k}  C_{rj} \left(\frac{-1}{\alpha_{j}}\right)^{r} B\left( 2k-s ; r-(2k-s)\right) \left( i\alpha_{j}\right)^{2k-s}.
 	\end{align}
 	Moreover,    $ \Phi(s)$     satisfies the functional equation 
 	\begin{eqnarray}
 	\Phi \left( 2k-s\right) -i^{2k}\Phi\left( s\right) = R \left( s\right), ~~where
 	\end{eqnarray}
 	\begin{align*}
 	R(s)=E^{B}(2k-s)-i^{2k} E^{B} (s) 
 	\end{align*}
	and  ${_{2}F_{1}}\left[a,b,c;z \right]$ is the hypergeometric function and   $B(a,b)$   is the Beta function.
 \end{thm}
   
 The proof  of Theorem \ref{ThM1} is similar to that of Hawkins and Knopp   \cite{HK1992}, with appropriate modifications for the more general group $\lambda_p=2\cos\left(\frac{\pi}{p}\right)>0$. 
 
  \begin{cor}\label{Cr1}
 	Suppose    $\Phi(s),~D^{0}(s),~D(s),~E^{0}(s),~E^{H}(s)$   and   $E^{B}(s)$   are    given as in Theorem \ref{ThM1}.   Then 
 	\begin{enumerate}
 		\item[(a)] $\Phi(s)$   is bounded uniformly in   $\sigma$    in lacunary vertical strips of the form
 		$$ S=\lbrace s=\sigma+it : 2k-\delta\leq\sigma\leq\delta; |t|\geq t_{o}>0 \rbrace.$$
 	\end{enumerate}
 	\begin{enumerate}
 		\item[(b)] $ \delta$ in $(a)$ can be chosen so that the poles of $\Phi(s)$   lying with in the lines   $s=(2k-\delta)+it$   and    $s=\delta+it$ are listed below in the sets;
 	\end{enumerate}
 	$$S_{0}=\lbrace 0,2k\rbrace,~~ S_{E_{0}}=\lbrace2k-L,2k-L+1,...,k-1,k,k+1,...,2k,...,L\rbrace, $$
 	$$ S_{H}=\lbrace [2k-\delta],...,0\rbrace, ~~S_{B}=\lbrace [2k-\delta],...,2k-L,...,2k-1\rbrace.$$
 	
 	The poles of $\Phi(s)$ in  each set arise from   $D^{0}(s), ~E^{0}(s), ~E^{H}(s)$   and   $E^{B}(s)$   respectively.
 	\begin{enumerate}
 		\item[(c)] The residues of    $\Phi(s)$   are given by the formula:
 	\end{enumerate} 
 	\begin{align}
 	\def\res{\mathop{Res}} \underset{s\in S_{0}}{Res}[D^{0}(s)]&=a_{0}\left( i^{2k}-1 \right),\label{EQ0}\\
 	\underset{s\in S_{E^{0}}} {Res}[E^{0}(s)]&= \sum_{m=k}^{L}C_m\left\lbrace -(-i)^{m}+i^{2k-m} \right\rbrace,\label{EQ1}  \\
 	\underset{s\in S_{H}}{Res}[E^{H}(s)]&=-\sum_{j=1}^{p}\sum_{r=1}^{k}C_{rj}\left\lbrace \sum_{m=0}^{[\delta]-2k}\frac{\Gamma(r+m)}{\Gamma(r)}\frac{(-1)^{r}}{m!}i^{m}\alpha_{j}^{-r-m}\right\rbrace ,\label{EQ2} \\
 	\underset{ s\in S_{B}}{Res}[E^{B}(s)]&=\sum_{j=1}^{p}\sum_{r=1}^{k}C_{rj}\left\lbrace \sum_{m=0}^{[\delta]-r}\frac{\Gamma(r+m)}{\Gamma(r)}\frac{(-1)^{r+m}}{m!}i^{m+r-2k}\alpha_{j}^{m}\right\rbrace.\label{EQ3}
 	\end{align}
 \end{cor}

Before we state and prove our main results, we state Perron's formula as Lemma 2.3 below (see   \cite{CN} for details.) We shall also use the convention of writing $\int_{(b)}$~ for ~$\int_{b-i\infty}^{b+\infty}.$
 \begin{lem}\label{LemPe}
  	Let    $\sigma_{0}$   be  the abscissa of absolute convergence for    $\displaystyle \varphi(s)=\sum_{m=1}^{\infty}a_{m}\lambda_{m}^{-s}$ and $\lbrace \lambda_{m} \rbrace$ be a  sequence of positive real numbers tending to $\infty$ as $m\rightarrow \infty$. Then for    $k\geq 0,\sigma>0$   and   $\sigma>\sigma_{0}$, 
  	\begin{equation}
  	\frac{1}{\Gamma(k+1)}{\sum_{\lambda_{m}\leq x}}^{_\prime} a_{m}\left( x-\lambda_{m}\right)^{k}=\frac{1}{2\pi i}\int_{(\sigma)}\frac{\Gamma(s)\varphi(s)x^{s+k}}{\Gamma(s+k+1)}ds,
  	\end{equation}
  	where the prime $'$ on the summation sign indicates that if    $k=0$   and   $x=\lambda_{m}$   for some positive integer   $m$,    then we count only   $\frac{1}{2}a_{m}$.
  \end{lem}
  The evaluation of  the integral  in (2.18) of Lemma 2.3, we consider a positively oriented rectangular contour formed by 
    $$[(2k-\sigma)-iT, \sigma-iT], \  [\sigma-iT, \sigma+iT],  \  [\sigma+iT, (2k-\sigma)+iT], \ [(2k-\sigma)+iT, (2k-\sigma)-iT],$$   use Stirling's approximation formula for the gamma function, Phragmen-Lindel$\ddot{o}$f theorem \cite{BC} and apply  Cauchy Residue theorem where in all cases the parameters are choosen appropriately to satisfy the conditions. Note that we use this in several places and cite it as  \cite{BC}
 \maketitle
\def\theequation{\thesection.\arabic{equation}} 
 \section{Main Results}
\setcounter{equation}{0}

 In this section, we shall use  the  techniques of Chandrasekharan and Narasimhan in \cite{CN} to  extend the first result in \cite{SISANN} to entire automorphic integrals on discrete Hecke  groups $G(\lambda).$ \\

\begin{thm}\label{thmm2} $\left(First Equivalence\right).$ Let   $\Phi(s)$   and   $R(s)$ be   as in    Theorem \ref{ThM1}.   Then the functional equation
	\begin{equation}\label{thmmIl}
	\Phi(2k-s)-i^{2k}\Phi(s)=R(s)
	\end{equation}
	is equivalent to the identity 
	\begin{equation}\label{thmmI}
	\frac{1}{\Gamma(\rho+1)}{\sum_{0\leq m\leq x}}^\prime a_m(x-m)^{\rho}= \Lambda_1(x)+  \Lambda_2(x)+ \Lambda_3(x)+ \Lambda_4(x)+ \Lambda_5(x),~~ where
		\end{equation}
	  
	\begin{eqnarray*} \Lambda_1(x) &=&i^{-2k}\left(\frac{2\pi}{\lambda} \right)^{-\rho} \sum_{m= 1}^{\infty}a_m\left( \frac{x}{m}\right) ^{\frac{\rho+2k}{2}}J_{\rho+2k}\left(\frac{4\pi\sqrt{mx}}{\lambda}\right) \nonumber  \\
		\Lambda_2(x)&=& i^{2k}\left( \frac{2\pi}{\lambda}\right)^{2k}\frac{a_{0}}{\Gamma(2k+\rho+1)}x^{2k+\rho} \nonumber \\
		\Lambda_3(x)&=& -\sum_{j=1}^{p}\sum_{r=1}^{M_{j}}C_{rj}\left( \frac{-1}{\alpha}_{j}\right)^{r}\frac{(i\alpha_{j})^{r}\left( \frac{2\pi}{\lambda}\right)^{r}x^{r+\rho}}{\Gamma(r+\rho+1)}{_{1}F_{1}}\left( r,r+\rho+1;.\frac{-i\alpha_{j}2\pi x}{\lambda}\right)\nonumber \\  
		\Lambda_4(x)&=& \sum_{j=1}^{p}\sum_{r=1}^{M_{j}}C_{rj}\left( \frac{-1}{\alpha}_{j}\right)^{r}\frac{(i)^{-2k}\left( \frac{2\pi}{\lambda}\right)^{2k}x^{2k+\rho}}{\Gamma(2k+\rho+1)} {_{1}F_{1}}\left( r,2k+\rho+1;\frac{-2\pi x}{i\alpha_{j}\lambda}\right)\nonumber  \\
		\Lambda_5(x)&=& -\sum_{m=k}^{L}\left\lbrace \frac{(-i)^{m}\left( \frac{2\pi}{\lambda}\right)^{m}x^{m+\rho}}{\Gamma(m+\rho+1)}-\frac{i^{2k-m}\left(\frac{2\pi}{\lambda}\right)^{2k-m}x^{2k-m+\rho}}{\Gamma(2k-m+\rho+1)}\right\rbrace, \end{eqnarray*}
	$x>0$,    $ \rho\geq 2\beta-2k-\frac{1}{2}$, and    $\beta$   is a number for which $\displaystyle \sum_{m=1}^{\infty}\frac{\vert a_m\vert}{m^{\beta}}<\infty.$
\end{thm}

   \begin{proof}
   We use \cite{BC} to arrive at
   \begin{align}
   & \frac{1}{2\pi i}\int_{(\sigma)}  \frac{\Gamma(s)\varphi(s)x^{s+\rho}}{\Gamma(s+\rho+1)}ds=\sum_{s\in\text{ Pole Set}} \text{Res} \left\{ \frac{\Gamma(s)\varphi(s) x^{s + \rho}}{\Gamma(s +
  	\rho + 1)} \right\}.\label{res1}
  	\end{align}
   Hence by using again \cite{BC}~$(\ref{res1})$ can now be written as 	
  	
  \begin{equation}\label{twoints} \frac{1}{2\pi i}\displaystyle\int_{(\sigma)}\frac{\left( \frac{2\pi}{\lambda}\right)^{s}\Phi(s)x^{s+\rho}}{\Gamma(s+\rho+1)}ds= A_1(x) + A_2(x),~~~where \end{equation}

  \begin{eqnarray}A_{1}(x)&=&\frac{1}{2\pi i}\int_{(2k-\sigma)}\frac{\left( \frac{2\pi}{\lambda}\right)^{s}\Phi(s)x^{s+\rho}}{\Gamma(s+\rho+1)}ds \label{A1}\\  
  A_{2}(x)&=& \sum \text{Res}\Bigg\lbrace \frac{\left( \frac{2\pi}{\lambda}\right)^{s}\Phi(s)x^{s+\rho}}{\Gamma(s+\rho+1)}\Bigg\rbrace .\label{A2}
  \end{eqnarray} 
  	We now show that the functions $A_1(x)$ and $A_2(x)$ can be expressed respectively as 
	
  	\begin{align}\label{A21}
  	A_{1}(x)&=i^{-2k}\left( \frac{2\pi}{\lambda}\right)^{-\rho}\sum_{m=1}^{\infty}a_{m}\left( \frac{x}{m}\right)^{\frac{2k+\rho}{2}}J_{2k+\rho}\left( \frac{4\pi\sqrt{mx}}{\lambda}\right)\nonumber \\
  	&+\sum_{j=1}^{p}\sum_{r=1}^{M_{j}}C_{rj}\left( \frac{-1}{\alpha_{j}}\right)^{r}\sum_{m=0}^{[\delta-2k]}\frac{(-1)^{m}}{m!}\frac{\Gamma(m+r)}{\Gamma(r)}\frac{\left( \frac{2\pi i\alpha_{j}}{\lambda x}\right)^{-m}x^{\rho}}{\Gamma(-m+\rho+1)}\nonumber \\
  	&-\sum_{j=1}^{p}\sum_{r=1}^{M_{j}}C_{rj}\left( \frac{-1}{\alpha_{j}}\right)^{r}\frac{\left( \frac{-2\pi xi}{\lambda }\right)^{r}x^{\rho}}{\Gamma(r+\rho+1)}{_{1}F_{1}}\left( r,\rho+r+1;\frac{-2\pi \alpha_{j}x}{\lambda}\right)\nonumber \\
  	&-\sum_{j=1}^{p}\sum_{r=1}^{M_{j}}C_{rj}\left( \frac{-1}{\alpha_{j}}\right)^{r}\sum_{m=0}^{[\delta-r]}\frac{(-1)^{m}}{m!}\frac{\Gamma(m+r)}{\Gamma(r)}\frac{i^{m+r-2k}\alpha_{j}^{m+r}\left( \frac{2\pi x}{\lambda }\right)^{2k-m-r}x^{\rho}}{\Gamma(2k-m-r+\rho+1)}\nonumber \\
  	&+\sum_{j=1}^{p}\sum_{r=1}^{M_{j}}C_{rj}\left( \frac{-1}{\alpha_{j}}\right)^{r}\frac{\left( \frac{2\pi x}{\lambda i}\right)^{2k}x^{\rho}}{\Gamma(2k+\rho+1)}{_{1}F_{1}}\left( r,\rho+2k+1;\frac{- 2\pi x}{i\lambda \alpha_{j}}\right).
  	\end{align} 
  	\begin{align}\label{A22}
  	A_{2}(x)&=\frac{a_{0}i^{2k}x^{2k+\rho}(\frac{2\pi}{\lambda})^{2k}}{\Gamma(2k+\rho+1)}-\frac{a_{0}x^{\rho}}{\Gamma(\rho+1)} \nonumber \\
  	&+ \sum_{m=k}^{L} C_{m}\Bigg\lbrace \frac{-(-i)^{m}(\frac{2\pi}{\lambda})^{m}x^{m+\rho}}{\Gamma(m+\rho+1)}+\frac{(i)^{2k-m}(\frac{2\pi}{\lambda})^{2k-m}x^{2k-m+\rho}}{\Gamma(2k-m+\rho+1)}\Bigg\rbrace \nonumber \\
  	&-\sum_{j=1}^{p}\sum_{r=1}^{M_{j}}C_{rj}\Bigg\lbrace\sum_{m=0}^{[\delta]-2k}\frac{\Gamma(m+r)}{\Gamma(r)}\frac{(-1)^{m}}{m!}i^{m}\alpha_{j}^{-r-m}\frac{(\frac{2\pi}{\lambda})^{-m}x^{-m+\rho}}{\Gamma(-m+\rho+1)}\Bigg\rbrace \nonumber \\
  	&+\sum_{j=1}^{p}\sum_{r=1}^{M_{j}}C_{rj}\Bigg\lbrace\sum_{m=0}^{[\delta]-r}\frac{\Gamma(m+r)}{\Gamma(r)}\frac{(-1)^{m+r}}{m!}i^{m+r-2k}\alpha_{j}^{m}\frac{(\frac{2\pi}{\lambda})^{2k-m-r}x^{2k-m-r+\rho}}{\Gamma(2k-m-r+\rho+1)}\Bigg\rbrace.
  	\end{align}
  	
  Observe that using the  functional equation $\Phi(2k-s)-i^{2k}\Phi(s)=R(s)$, we have  
  \[  A_{1}(x)= \frac{1}{2\pi i}\int_{(2k-\sigma)} \frac{\left(\frac{2\pi}{\lambda}\right)^{s}i^{-2k}\Phi(2k-s)x^{s+\rho}}{\Gamma(s+\rho+1)}ds- \frac{1}{2\pi i}\int_{(2k-\sigma)} \frac{\left(\frac{2\pi}{\lambda}\right)^{s}i^{-2k}R(s)x^{s+\rho}}{\Gamma(s+\rho+1)}ds.\]
      Since $\Phi(s)=\left( \frac{2\pi}{\lambda}\right)^{-s}\Gamma(s)\varphi(s),$ and denoting the first integral by $I(x)$ we see that  
  \begin{align*}
  I(x)&=\frac{1}{2\pi i}\int_{(2k-\sigma)}\frac{\left(\frac{2\pi}{\lambda} \right)^{2s-2k}i^{-2k}\Gamma(2k-s)\varphi(2k-s)x^{s+\rho}}{\Gamma(s+\rho+1)}ds\\
  &=\frac{1}{2\pi i}\int_{(\sigma)}\frac{\left(\frac{2\pi}{\lambda} \right)^{-2s+2k}i^{-2k}\Gamma(s)\varphi(s)x^{-s+\rho+2k}}{\Gamma(2k-s+\rho+1)}ds \\
  &= i^{-2k}\left( \frac{2\pi}{\lambda}\right)^{2k}x^{2k+\rho}\frac{1}{2\pi i}\sum_{m=1}^{\infty}a_{m}\int_{(\sigma)}\frac{\Gamma(s)\left( \frac{4\pi^{2}x m}{\lambda^{2}}\right)^{-s}}{\Gamma(2k-s+\rho+1)}ds,
  \end{align*} where we have used change of  variable from  $s$   to   $2k-s ,$  in the first integral.   Letting    $\frac{w}{2}=s, ~  \nu=2k+\rho,$   and simplifying expressions, we obtain
      \begin{align*}
  I(x)&=i^{-2k}\left( \frac{2\pi}{\lambda}\right)^{\nu-\rho} 2^{\nu} x^{\nu} \sum_{m=1}^{\infty}a_{m}\frac{1}{2\pi i} \int_{(2\sigma)}\frac{\Gamma(\frac{w}{2})\left( \frac{4\pi \sqrt{mx}}{\lambda}\right)^{-w}}{\Gamma(\nu-\frac{w}{2}+1)}2^{w-\nu-1}dw.\\
&=  i^{-2k}\left( \frac{2\pi}{\lambda}\right)^{-\rho}\sum_{m=1}^{\infty}a_{m}\left( \frac{x}{m}\right)^{\frac{2k+\rho}{2}}J_{2k+\rho}\left(\frac{4\pi\sqrt{mx}}{\lambda} \right),
   \end{align*} provided that   $\rho>\sigma-2k$   and    $\sigma>2k,~~   k\in\mathbb{Z}$. Note also that we have used the definition of the Bessel J-function here. Now we evaluate the second integral denoting by $H(x)$ and using the expression for $R(x)$ as follows $$R(s)=i^{2k}\sum_{j=1}^{P}\sum_{r=1}^{M_{j}}C_{ri}\left( \frac{-1}{\alpha_{j}}\right)^{r}\bigg\lbrace\left(i\alpha_{j} \right)^{s}B(s,r-s)-i^{-s}\alpha_{j}^{2k-s}B(2k-s,r-(2k-s))\bigg\rbrace, $$ 
     $$ H(x)=\frac{1}{2\pi i}\int_{(2k-\sigma)}\frac{\left( \frac{2\pi}{\lambda}\right)^{s}i^{-2k}R(s)x^{s+\rho}}{\Gamma(s+\rho+1)}ds $$
         \begin{align*}
   H(x)=\sum_{j=1}^{P}\sum_{r=1}^{M_{j}}C_{ri}\left( \frac{-1}{\alpha_{j}}\right)^{r}\Bigg\lbrace \frac{1}{2\pi i}\int_{(2k-\sigma)}\frac{\left( \frac{2\pi i\alpha_{j}}{\lambda}\right)^{s} B(s,r-s)x^{s+\rho}}{\Gamma(s+\rho+1)}ds\\
  -\alpha_{j}^{2k-s}\frac{1}{2\pi i}\int_{(2k-\sigma)}\frac{\left( \frac{2\pi }{\lambda i}\right)^{s} B(2k-s,r-(2k-s))x^{s+\rho}}{\Gamma(s+\rho+1)}ds\Bigg\rbrace.
  \end{align*} 
Using the properties of the beta function and after some algebraic manipulations, we have  
  	\begin{align*}
  	\frac{1}{2\pi i}\int_{(2k-\sigma)}\frac{\left( \frac{2\pi i\alpha_{j}}{\lambda}\right)^{s} B(s,r-s)x^{s+\rho}}{\Gamma(s+\rho+1)}ds
  	&=-\sum_{m=0}^{[\sigma]-2k}\frac{(-1)^{m}\left( \frac{2\pi i\alpha_{j}}{\lambda}\right)^{-m}x^{-m+\rho}\Gamma(m+r)}{m!\Gamma(-m+\rho+1)\Gamma(r)}\\
  	&-\sum_{m=0}^{\infty}\frac{(-1)^{m}\left( \frac{2\pi i\alpha_{j}}{\lambda}\right)^{m+r}x^{m+r+\rho}\Gamma(m+r)}{m!\Gamma(m+r+\rho+1)\Gamma(r)}\\
  	&=-\sum_{m=0}^{[\sigma]-2k}\frac{(-1)^{m}\left( \frac{2\pi i\alpha_{j}}{\lambda}\right)^{-m}x^{-m+\rho}\Gamma(m+r)}{m!\Gamma(-m+\rho+1)\Gamma(r)}\\
  	&-\frac{\left( \frac{2\pi i\alpha_{j}x}{\lambda}\right)^{r}x^{\rho}}{\Gamma(\rho+r+1)}{_{1}F_{1}}\left(r,\rho+r+1; -\frac{2\pi i\alpha_{j}x}{\lambda} \right),~~and
  	\end{align*}
  	\begin{align*}
  	\frac{1}{2\pi i}\int_{(2k-\sigma)}\frac{\left( \frac{2\pi }{\lambda i}\right)^{s} B(2k-s,r-(2k-s))x^{s+\rho}}{\Gamma(s+\rho+1)}ds &=-\sum_{m=0}^{[\delta-r]}\left(\frac{2\pi}{i\lambda} \right)^{2k-m-r}\frac{\alpha_{j}^{m+r}}{\Gamma(2k-m-r+\rho+1)}\frac{(-1)^{m}}{m!}\frac{\Gamma(m+r)}{\Gamma(r)}x^{2k-m-r+\rho}\\
  	&-\sum_{m=0}^{\infty}\left(\frac{2\pi}{i\lambda} \right)^{2k+m}\frac{\alpha_{j}^{-m}}{\Gamma(2k+m+\rho+1)}\frac{(-1)^{m}}{m!}\frac{\Gamma(m+r)}{\Gamma(r)}x^{2k+m+\rho}.
  	\end{align*} Using the properties of hypergeometric series and simplifying, we get \begin{align*}
  H(x)=&-\sum_{j=1}^{p}\sum_{r=1}^{M_{j}}C_{rj}\left( \frac{-1}{\alpha_{j}}\right)^{r}\sum_{m=0}^{[\delta-2k]}\frac{(-1)^{m}}{m!}\frac{\Gamma(m+r)}{\Gamma(r)}\frac{\left( \frac{2\pi i\alpha_{j}}{\lambda x}\right)^{-m}x^{\rho}}{\Gamma(-m+\rho+1)}\\
  &+\sum_{j=1}^{p}\sum_{r=1}^{M_{j}}C_{rj}\left( \frac{-1}{\alpha_{j}}\right)^{r}\frac{\left( \frac{-2\pi xi}{\lambda }\right)^{r}x^{\rho}}{\Gamma(r+\rho+1)}{_{1}F_{1}}\left( r,\rho+r+1;\frac{-2\pi \alpha_{j}x}{\lambda}\right)\\
  &+\sum_{j=1}^{p}\sum_{r=1}^{M_{j}}C_{rj}\left( \frac{-1}{\alpha_{j}}\right)^{r}\sum_{m=0}^{[\delta-r]}\frac{(-1)^{m}}{m!}\frac{\Gamma(m+r)}{\Gamma(r)}\frac{i^{m+r-2k}\alpha_{j}^{m+r}\left( \frac{2\pi x}{\lambda }\right)^{2k-m-r}x^{\rho}}{\Gamma(2k-m-r+\rho+1)}\\
  &-\sum_{j=1}^{p}\sum_{r=1}^{M_{j}}C_{rj}\left( \frac{-1}{\alpha_{j}}\right)^{r}\frac{\left( \frac{2\pi x}{\lambda i}\right)^{2k}x^{\rho}}{\Gamma(2k+\rho+1)}{_{1}F_{1}}\left( r,\rho+2k+1;\frac{- 2\pi x}{i\lambda \alpha_{j}}\right). 
  \end{align*}
Since    $A_{1}(x)=I(x)-H(x)$,   we have proved (\ref{A21}) and proceed to computing $A_{2}(x)$ as follows.
          $$A_{2}(x)=\sum_{s\in polset~ of\Phi(s)}\text{Res}\Bigg\lbrace \frac{\left( \frac{2\pi}{\lambda}\right)^{s}\Phi(s)x^{s+\rho}}{\Gamma(s+\rho+1)}\Bigg\rbrace.$$
 From theorem $\ref{ThM1},$~  $\Phi(s)$ has been expressed in terms of $D^{0}(s),~E^{0}(s),~E^{H}(s)$, and $E^{B}(s)$. Thus $A_2$ can be expressed as follows
  \begin{align*}
  A_{2}(x)=\bigg\lbrace Res[D^{0}(s)]+Res[E^{0}(s)]+Res[E^{H}(s)]+Res[E^{B}(s)]\bigg\rbrace \frac{\left( \frac{2\pi}{\lambda}\right)^{s}x^{s+\rho}}{\Gamma(s+\rho+1)}.
  \end{align*}
  Using formula $(\ref{EQ0}),(\ref{EQ1}),(\ref{EQ2})$and $(\ref{EQ3})$ we obtain the following and hence (\ref{A22}) is proved.
  
  \begin{align*}
  A_{2}(x)&=\frac{a_{0}i^{2k}x^{2k+\rho}(\frac{2\pi}{\lambda})^{2k}}{\Gamma(2k+\rho+1)}-\frac{a_{0}x^{\rho}}{\Gamma(\rho+1)} \nonumber \\
  &+ \sum_{m=k}^{L} C_{m}\Bigg\lbrace \frac{-(-i)^{m}(\frac{2\pi}{\lambda})^{m}x^{m+\rho}}{\Gamma(m+\rho+1)}+\frac{(i)^{2k-m}(\frac{2\pi}{\lambda})^{2k-m}x^{2k-m+\rho}}{\Gamma(2k-m+\rho+1)}\Bigg\rbrace \nonumber \\
  &-\sum_{j=1}^{p}\sum_{r=1}^{M_{j}}C_{rj}\Bigg\lbrace\sum_{m=0}^{[\delta]-2k}\frac{\Gamma(m+r)}{\Gamma(r)}\frac{(-1)^{m}}{m!}i^{m}\alpha_{j}^{-r-m}\frac{(\frac{2\pi}{\lambda})^{-m}x^{-m+\rho}}{\Gamma(-m+\rho+1)}\Bigg\rbrace \nonumber \\
  &+\sum_{j=1}^{p}\sum_{r=1}^{M_{j}}C_{rj}\Bigg\lbrace\sum_{m=0}^{[\delta]-r}\frac{\Gamma(m+r)}{\Gamma(r)}\frac{(-1)^{m+r}}{m!}i^{m+r-2k}\alpha_{j}^{m}\frac{(\frac{2\pi}{\lambda})^{2k-m-r}x^{2k-m-r+\rho}}{\Gamma(2k-m-r+\rho+1)}\Bigg\rbrace
  \end{align*}      Finally,   we  now express the integral in the right-hand side of (3.4) as
  \begin{align*}
  \frac{1}{2\pi i}\int_{(\delta)}\frac{\Gamma(s)\varphi(s)x^{s+\rho}}{\Gamma(s+\rho+1)}ds=A_{1}(x)+A_{2}(x)= \left( I(x)-H(x)\right)+A_{2}(x),
  \end{align*}
  and using the expressions for $ A_{1}(x)$ and $A_{2}(x).$ [See $(\ref{A21})$ and $(\ref{A22})$ respectively] we obtain
  \begin{align}\label{Iden2}
  \frac{1}{2\pi i}\int_{(\delta)}\frac{\Gamma(s)\varphi(s)x^{s+\rho}}{\Gamma(s+\rho+1)}ds &=  i^{-2k}\left(\frac{2\pi}{\lambda} \right)^{-\rho} \sum_{m= 1}^{\infty}a_m\left( \frac{x}{m}\right) ^{\frac{\rho+2k}{2}}J_{\rho+2k}\left(\frac{4\pi\sqrt{mx}}{\lambda}\right) \nonumber  \\
  &+i^{2k}\left( \frac{2\pi}{\lambda}\right)^{2k}\frac{a_{0}}{\Gamma(2k+\rho+1)}x^{2k+\rho}-\frac{a_{0}x^{\rho}}{\Gamma(\rho+1)} \nonumber \\
  & -\sum_{j=1}^{p}\sum_{r=1}^{M_{j}}C_{rj}\left( \frac{-1}{\alpha}_{j}\right)^{r}\frac{(i\alpha_{j})^{r}\left( \frac{2\pi}{\lambda}\right)^{r}x^{r+\rho}}{\Gamma(r+\rho+1)}{_{1}F_{1}}\left( r,r+\rho+1;.\frac{-i\alpha_{j}2\pi x}{\lambda}\right)\nonumber \\  
  & +\sum_{j=1}^{p}\sum_{r=1}^{M_{j}}C_{rj}\left( \frac{-1}{\alpha}_{j}\right)^{r}\frac{(i)^{-2k}\left( \frac{2\pi}{\lambda}\right)^{2k}x^{2k+\rho}}{\Gamma(2k+\rho+1)} {_{1}F_{1}}\left( r,2k+\rho+1;\frac{-2\pi x}{i\alpha_{j}\lambda}\right)\nonumber  \\
  &-\sum_{m=k}^{L}\left\lbrace \frac{(-i)^{m}\left( \frac{2\pi}{\lambda}\right)^{m}x^{m+\rho}}{\Gamma(m+\rho+1)}-\frac{i^{2k-m}\left(\frac{2\pi}{\lambda}\right)^{2k-m}x^{2k-m+\rho}}{\Gamma(2k-m+\rho+1)}\right\rbrace .
  \end{align}
  Therefore,  for $\rho\geq 0,$~ and~$\frac{\rho+2k}{2}>\delta$ we get the identity  $(\ref{thmmI})$. \\
  
  Now we prove the converse of the theorem. To this end suppose  $F(z)$ is an entire automorphic integral with  a Fourier series expansion for $z\in\mathcal{H}$ and  satisfies the relation 
  $$ F(z)=\sum_{m=0}^{\infty} a_{m}e^{2\pi imz/\lambda},$$  
  \begin{equation}\label{con}
  z^{-2k}F\left(\frac{-1}{z} \right)=F(z)+q(z),
  \end{equation}
  where $q(z)$ is  the rational period function given by Lemma $(\ref{LEMA2})$. Then by $(\ref{con})$~ and the Fourier expansion of $F$ we have $$\displaystyle z^{-2k}\sum_{m=0}^{\infty} a_{m}e^{-2\pi im/\lambda z}=\sum_{m=0}^{\infty} a_{m}e^{2\pi imz/\lambda}+q(z).$$ Letting $z=\frac{iy\lambda}{2\pi}, ~~y>0,$ then we get
  \begin{align}\label{con2}
  \left(\frac{2\pi}{iy\lambda} \right)^{2k}\sum_{m=0}^{\infty}a_{m}e^{\frac{-4\pi^{2}m}{y\lambda^{2}}}=\sum_{m=0}^{\infty}a_{m}e^{-my} +q\left( \frac{iy\lambda}{2\pi}\right).
  \end{align}
  To prove the converse it suffices to show that $(\ref{iden2})$ implies $(\ref{con2}).$ To this end we consider six integrals defining $L_{1}(y)\cdots,L_{6}(y),$ corresponding to the six expressions occurring in $(\ref{Iden2})$ . We evaluate all the integrals by interchanging integration and summation which can be justified.   
  \begin{align*}
  L_{1}(y)&=\int_{0}^{\infty}\Bigg\lbrace\frac{1}{\Gamma(\rho+1)}\sum_{0\leq m\leq x} a_{m}(x-m)^{\rho}\Bigg\rbrace y^{\rho+1}e^{-xy}dx,\\
  &=\frac{a_{0}y^{\rho+1}}{\Gamma(\rho+1)}\int_{0}^{\infty}e^{-xy}x^{\rho}dx+\frac{y^{\rho+1}}{\Gamma(\rho+1)}\sum_{1\leq m\leq x} a_{m}\int_{0}^{\infty} (x-m)^{\rho}e^{-xy}dx,\\
  &=a_{0}+\sum_{m=1}^{\infty}a_{m}e^{-my} ~~~ \rho \geq 0. 
  \end{align*}
  Similarly for $\frac{\rho+2k}{2}\leq \beta $ we get
  \begin{align*}
  L_{2}(y)&= \int_{0}^{\infty} i^{-2k}\left(\frac{2\pi}{\lambda} \right)^{-\rho}\sum_{m=0}^{\infty}a_{m}\left(\frac{x}{m} \right)^{\frac{2k+\rho}{2}}J_{2k+\rho}\left(\frac{4\pi\sqrt{mx}}{\lambda} \right)y^{\rho+1}e^{-xy}dx,\\
  &=i^{-2k}\left(\frac{2\pi}{\lambda} \right)^{-\rho}y^{\rho+1}\sum_{m=0}^{\infty}a_{m}m^{\frac{-(2k+\rho)}{2}}\Bigg\lbrace \frac{2(4\pi\sqrt{m}/\lambda)^{2k+\rho}}{(2y)^{2k+\rho+1}}e^-{\frac{(4\pi\sqrt{m})^{2}}{4y\lambda^{2}}} \Bigg\rbrace.
  \end{align*}
  Thus with simple algebraic manipulations we obtain
  \begin{align*}
  L_{2}(y)=i^{-2k}\left(\frac{2\pi}{\lambda y} \right)^{2k}\sum_{m=0}^{\infty}a_{m}e^{-\frac{4\pi^{2}m}{y\lambda^{2}}}.
  \end{align*}
  \begin{align*}
  L_{3}(y)=\int_{0}^{\infty} e^{-xy}y^{\rho+1}\Bigg\lbrace i^{2k}\left(\frac{2\pi}{\lambda y} \right)^{2k}\frac{a_{0}}{\Gamma(2k+\rho+1)}x^{2k+\rho}-\frac{a_{0}x^{\rho}}{\Gamma(\rho+1)} \Bigg\rbrace dx. 
  \end{align*}
  Using integration  by substitution and the standard integral representation of $\Gamma(s)$~  we get
  \begin{align*}
  L_{3}(y)=a_{0}i^{2k}\left(\frac{2\pi}{\lambda y} \right)^{2k}-a_{0}.
  \end{align*}
  \begin{align*}
  L_{4}(y)=\int_{0}^{\infty}\sum_{m=k}^{L}C_{m}y^{\rho+1}e^{-xy}\Bigg\lbrace \frac{(-i)^{m}\left(2\pi/\lambda \right)^{m}x^{m+\rho}}{\Gamma(m+\rho+1)} -\frac{(i)^{2k-m}\left(2\pi/\lambda \right)^{2k-m}x^{2k-m+\rho}}{\Gamma(2k-m+\rho+1)}   \Bigg\rbrace dx.
  \end{align*}
 After evaluating and Simplifying we obtain  
  \begin{align*}
  L_{4}(y)=\sum_{m=k}^{L}C_{m}\Bigg\lbrace \left( \frac{2\pi}{i \lambda y}\right)^{m}-(-1)^{m}\left( \frac{2\pi}{i\lambda y}\right)^{2k-m}\Bigg\rbrace.
  \end{align*}
  \begin{align*}
  L_{5}(y)=\int_{0}^{\infty} {_{1}F_{1}}\left(r,\rho+r+1;\frac{-2\pi i\alpha_{j}x}{\lambda} \right)\frac{x^{\rho+r}y^{\rho+1}e^{-xy}}{\Gamma(r+\rho+1)}dx.
  \end{align*}
 Using the series representation of the hypergeometric function for $\lambda y>2\pi\alpha_{j},$  we obtain
  $$L_{5}(y)=\sum_{m=0}^{\infty}\frac{\Gamma(m+r)}{\Gamma(r)}\frac{(-1)^{m}}{m!}\left( \frac{i\alpha_{j}2\pi}{\lambda y}\right)^{m}y^{-r}.$$ This series converges absolutely for $\lambda y> 2\pi\alpha_{j}.$
 In a Similar way, we compute $L_{6}$ as follows 
  $$L_{6}(y)=\int_{0}^{\infty} {_{1}F_{1}}\left(r,\rho+2k+1;\frac{-2\pi x}{i\alpha_{j}\lambda} \right)\frac{x^{\rho+2k}y^{\rho+1}e^{-xy}}{\Gamma(2k+\rho+1)}dx = \sum_{m=0}^{\infty}\frac{\Gamma(m+r)}{\Gamma(r)}\frac{(-1)^{m}}{m!}\left( \frac{2\pi}{i\lambda y\alpha_{j}}\right)^{m}y^{-2k}. $$
  This series converges absolutely for $y>\frac{2\pi}{\lambda\alpha_{j}}.$  Combining the results of the integrals for $L_{1}(y),L_{2}(y),\cdots,L_{6}(y)$ respectively we have 
  \begin{align*}L_{1}(y)&=L_{2}(y)+L_{3}(y)-L_{4}(y)-\sum_{j=1}^{p}\sum_{r=1}^{M_{j}}C_{rj}\left( \frac{-1}{\alpha_{j}}\right)^{r}(i\alpha_{j})^{r}\left(\frac{2\pi}{\lambda} \right)^{r}L_{5}(y)\\
  &+\sum_{j=1}^{p}\sum_{r=1}^{M_{j}}C_{rj}\left( \frac{-1}{\alpha_{j}}\right)^{r}(i)^{-2k}\left(\frac{2\pi}{\lambda} \right)^{2k} L_{6}(y).
  \end{align*}
  Thus after simple algebraic manipulations, we see that the identity in $(\ref{Iden2})$ implies
  \begin{align}\label{eqn22}
  \sum_{m=0}^{\infty}a_{m}e^{-my}&=\left(\frac{2\pi}{iy\lambda} \right)^{2k} \sum_{m=0}^{\infty}a_{m}e^{\frac{-4\pi^{2}m}{y\lambda^{2}}}+a_{0}i^{2k}\left(\frac{2\pi}{\lambda y} \right)^{2k}-a_{0},\nonumber \\
  &-\sum_{m=k}^{L}C_{m}\Bigg\lbrace \left( \frac{2\pi}{i \lambda y}\right)^{m}-(-1)^{m}\left( \frac{2\pi}{i\lambda y}\right)^{2k-m}\Bigg\rbrace,\nonumber \\
  & -\sum_{j=1}^{p}\sum_{r=1}^{M_{j}}C_{rj}\left( \frac{-1}{\alpha_{j}}\right)^{r}\Bigg\lbrace \alpha_{j}^{r}\left(\frac{2\pi i}{y\lambda} \right)^{r}\sum_{m=0}^{\infty}\frac{\Gamma(m+r)}{\Gamma(r)}\frac{(-1)^{m}}{m!}\left(\frac{2\pi i\alpha_{j}}{\lambda y} \right)^{m}\nonumber \\
  & -\left(\frac{2\pi }{iy\lambda} \right)^{2k}\sum_{m=0}^{\infty}\frac{\Gamma(m+r)}{\Gamma(r)}\frac{(-1)^{m}}{m!}\left(\frac{2\pi}{i\alpha_{j}\lambda y} \right)^{m} \Bigg\rbrace,
  \end{align} 
  provided that $\rho+2k+\frac{1}{2}\geq 2\beta$ and $\displaystyle y>\max_{1\leq j \leq p}\Bigg\lbrace \frac{2\pi\alpha_{j}}{\lambda},\frac{2\pi}{\alpha_{j}\lambda}\Bigg\rbrace .$\\
  Recall that the rational periodic  function in Lemma $\ref{LEMA1}.$ 
  $$q(z)=\sum_{k\leq r\leq L}C_{r}f_{r}(z,0)+\sum_{j=1}^{p}\sum_{r=1}^{M_{j}}C_{rj}f_{r}(z,\alpha_{j}),$$
  where $f_{r}(z,0)=z^{-r}-(-1)^{r}z^{-2k+r}$ and $f_{r}(z,\alpha_{j})=(z-\alpha_{j})^{-r}-(-1)^{r} \alpha_{j}^{-r} z^{-2k+r}\left(z+\frac{1}{\alpha_{j}} \right)^{-r}. $ Letting $z=\frac{iy\lambda}{2\pi}$ and applying the binomial expansion to $f_{r}(\frac{iy\lambda}{2\pi},\alpha_{j})$ after simplifying some steps we obtain
  \begin{align}\label{Rpf1}
  q\left(\frac{iy\lambda}{2\pi} \right)=&\sum_{r=1}^{L}C_{r}\Bigg\lbrace \left( \frac{2\pi}{iy\lambda}\right)^{r}-(-1)^{r}\left( \frac{2\pi}{iy\lambda}\right)^{2k-r}\Bigg\rbrace \nonumber\\
  &+\sum_{j=1}^{p}\sum_{r=1}^{M_{j}}C_{rj}\left( \frac{-1}{\alpha_{j}}\right)^{r}\Bigg\lbrace \alpha_{j}^{r}\left(\frac{2\pi i}{y\lambda} \right)^{r}\sum_{m=0}^{\infty}\frac{\Gamma(m+r)}{\Gamma(r)}\frac{(-1)^{m}}{m!}\left(\frac{2\pi i\alpha_{j}}{\lambda y} \right)^{m}\nonumber \\
  & -\left(\frac{2\pi }{iy\lambda} \right)^{2k}\sum_{m=0}^{\infty}\frac{\Gamma(m+r)}{\Gamma(r)}\frac{(-1)^{m}}{m!}\left(\frac{2\pi}{i\alpha_{j}\lambda y} \right)^{m} \Bigg\rbrace.
  \end{align}
  Since $F(z)$ has a Fourier expansion of the form $\displaystyle \sum_{m=0}^{\infty}a_{m}e^{2\pi imz/\lambda}$  and $q(\frac{iy\lambda}{2\pi})$ is  represented  $(\ref{Rpf1}), (\ref{eqn22})$ may be written as
  $$ F\left( \frac{iy\lambda}{2\pi} \right)=\left( \frac{iy\lambda}{2\pi} \right)^{-2k}F\left( \frac{-1}{\frac{iy\lambda}{2\pi}} \right)- q\left( \frac{iy\lambda}{2\pi} \right).$$ Hence by the identity theorem  the automorphic transformation
  $$z^{-2k}F\left(\frac{-1}{z} \right)=F(z)+q(z),$$ 
  follows for $z\in\mathcal{H}.$ This concludes the proof of the equivalence of the functional equation to the identity $(\ref{Iden2}).$
   
  \end{proof}

 \maketitle
\def\theequation{\thesection.\arabic{equation}} 
 
\setcounter{equation}{0}

 \begin{thm}\label{thm31}$\left(Second~ Equivalence\right)$ 
 	Let $\Phi(s)$ and $R(s)$ as in Theorem $\ref{ThM1}$ then the functional equation 
 	\begin{equation}\label{func2}
 	\Phi(2k-s)-i^{2k}\Phi(s)=R(s)
 	\end{equation}
 	is equivalent to the identity
 	\begin{align}\label{iden2}
 	&\left( \frac{-1}{y} \frac{d}{dy}\right)^{\rho}\left(\frac{1}{y}\sum_{m=1}^{\infty}a_{m}e^{-y\sqrt{m}} \right) 
 	=-\frac{2^{\rho}}{\sqrt{\pi}y^{2\rho+1}}a_{0}\Gamma\left( \rho+\frac{1}{2}\right)\nonumber \\
 	&+\frac{i^{-2k}}{\sqrt{\pi}}2^{4k+2\rho+1}\Gamma\left(2k+\rho+\frac{1}{2}\right)\left( \frac{2\pi}{\lambda}\right)^{2k}\sum_{m=0}^{\infty}\frac{a_{m}}{\left(y^{2}+4\left(\frac{2\pi}{\lambda}\right)^{2}m \right)^{2k+\rho+\frac{1}{2}}}\nonumber \\
 	&-\frac{2^{\rho}}{\sqrt{\pi}y^{2\rho+1}}\sum_{j=1}^{p}\sum_{r=1}^{M_{j}}C_{rj}\left( \frac{-1}{\alpha_{j}}\right)^{r}\Gamma\left(r+\rho+\frac{1}{2}\right)\Psi\left(r,-\rho+\frac{1}{2};\frac{\lambda y^{2}}{8\pi i\alpha_{j}} \right)\nonumber \\
 	&+\frac{2^{\rho}}{\sqrt{\pi}y^{2\rho+1}}\sum_{j=1}^{p}\sum_{r=1}^{M_{j}}C_{rj}\left( \frac{-1}{\alpha_{j}}\right)^{r}\left(\frac{i\alpha_{j}\lambda y^{2}}{8\pi } \right)^{r-2k}\alpha_{j}^{2k} \Gamma\left(2k+\rho+\frac{1}{2}\right) \nonumber \\
 	&\times \Psi\left(r,r-2k-\rho+\frac{1}{2};\frac{i\alpha_{j}\lambda y^{2}}{8\pi } \right)\nonumber\\
 	&-\frac{2^{\rho}}{\sqrt{\pi}y^{2\rho+1}}\sum_{m=k}^{L}\Bigg\lbrace \left(\frac{8\pi i}{\lambda y^{2}} \right)^{2k-m}\Gamma\left( 2k-m+\rho+1\right)-\left(\frac{-8\pi i}{\lambda y^{2}} \right)^{m}\Gamma\left( m+\rho+\frac{1}{2}\right)\Bigg\rbrace.
 	\end{align}
 	Provided that $Re(y)>0,~ \rho\in\mathbb{Z},~\rho\geq 0$ ~and~ $\rho+2k\geq\beta+\frac{1}{2},$~ where ~ $\sum_{m=1}^{\infty}\frac{|a_{m}|}{m^{\beta}}<\infty.$
 \end{thm}
 \begin{proof}
 By theorem $\ref{thmm2}$ the identity  $(\ref{thmmI})$ is equivalent to the functional equation $(\ref{func2}).$ Hence to prove this theorem it suffices to show that $(\ref{func2})$ implies $(\ref{iden2})$ and $(\ref{iden2})$ in turn implies $(\ref{func2}).$ 
 
 Now we first show $(\ref{func2})$ implies $(\ref{iden2}).$ To this end let $\varphi(s)=\displaystyle\sum_{n=1}^{\infty}\frac{a_n}{n^{s}}$~ with~ $\displaystyle \sum_{n=1}^{\infty}\frac{\vert a_n\vert}{n^{\beta}}<\infty$. Then by Lemma$\ref{LemPe},$ for $\rho\geq0,   \delta>0,\delta\geq\beta$ we have \begin{eqnarray}\label{PER}
 \frac{1}{\Gamma(\rho+1)}{\sum_{0\leq m \leq x}}^\prime a_{m}(x-m)^{\rho}=\frac{1}{2\pi i}\int_{\delta-i\infty}^{\delta+i\infty}\frac{\Gamma(s)\varphi(s)x^{s+\rho}}{\Gamma(s+\rho+1)}ds.
 \end{eqnarray}
 
 As in Chandrasekharan and Narasimhan (\cite{CN} page $9),$ we multiply~ $(\ref{PER})$ by~ $e^{-y\sqrt{x}}x^{-\frac{1}{2}}$  and integrating with respect to the variable $x$ from $x=0$ to $\infty;$ and further assuming $\delta>2k, $ where Re$s=\delta$ is the vertical path of integration. Now choose $\delta=\beta+P,$ where 
 $ P\in\mathbb{Z},$ and $P$ is large enough to guarantee $\delta>2k$ and $\delta\notin\mathbb{Z}$ to have
 \begin{align}
 \int_{0}^{\infty}e^{-y\sqrt{x}}x^{-\frac{1}{2}}\Bigg\lbrace\frac{1}{\Gamma(\rho+1)}\sum_{0\leq m\leq x}a_{m}(x-m)^{\rho}\Bigg\rbrace dx=\int_{0}^{\infty}e^{-y\sqrt{x}}x^{-\frac{1}{2}}\Bigg\lbrace\frac{1}{2\pi i} \int_{(\delta)}\frac{\varphi(s)\Gamma(s)x^{s+\rho}}{\Gamma(s+\rho+1)}ds\Bigg\rbrace dx,
 \end{align}
 with $\rho +2k \geq\delta+\frac{1}{2},$ and $y\in\mathbb{R^{+}}.$ 
 Chandrasekharan and Narasimhanin \cite{CN} for $\rho+2k+\frac{1}{2}\geq\beta$ and $\lambda_{n}$ sequence of positive real numbers; $\lambda_{n}\rightarrow\infty,$   showed the identity

 	\begin{align}\label{idencn}
 \sum_{n=1}^{\infty} a_{n}\int_{\lambda_{n}}^{\infty}\frac{(x-\lambda_{n})^{\rho}}{\Gamma(\rho+1)}e^{-y\sqrt{x}}x^{-\frac{1}{2}}dx=2(-2)^{\rho}\left( \frac{1}{y}\frac{d}{dy}\right)^{\rho}\Bigg[\frac{1}{y}\sum_{n=1}^{\infty}a_{n}e^{-y\sqrt{\lambda_{n}}}\Bigg].
 \end{align}
 Then put $\lambda_{n}=m$ in $(\ref{idencn})$ to have 
 \begin{align*}
 \sum_{m=1}^{\infty} a_{m}\int_{m}^{\infty}\frac{(x-m)^{\rho}}{\Gamma(\rho+1)}e^{-y\sqrt{x}}x^{-\frac{1}{2}}dx=2(-2)^{\rho}\left( \frac{1}{y}\frac{d}{dy}\right)^{\rho}\Bigg[\frac{1}{y}\sum_{m=1}^{\infty}a_{m}e^{-y\sqrt{m}}\Bigg].
 \end{align*}
 Since $\Phi(s)=\Gamma(s)\varphi(s)\left(\frac{2\pi}{\lambda}\right)^{-s},$ where Re$s>\beta$ and $\sum_{m=1}^{\infty}\frac{|a_{m}|}{m^{\beta}}<\infty,$  the right-hand side of $(\ref{PER})$ becomes
 \begin{align*}
 U(y)=\int_{0}^{\infty}e^{-y\sqrt{x}}x^{\frac{-1}{2}}\Bigg\lbrace \frac{1}{2\pi i}\int_{(\delta)}\left(\frac{2\pi}{\lambda} \right)^{s} \frac{\Phi(s) x^{s+\rho}}{\Gamma(s+\rho+1)}ds \Bigg\rbrace dx,
 \end{align*}
 for $\delta\geq\beta.$ Interchanging the order of integration  for $\rho\geq 0$ we obtain
 \begin{align*}
 U(y)&=\frac{1}{2\pi i}\int_{(\delta)}\left(\frac{2\pi}{\lambda} \right)^{s}\frac{\Phi(s)}{\Gamma(s+\rho+1)} ds \int_{0}^{\infty}e^{-y\sqrt{x}}x^{s+\rho-\frac{1}{2}}dx,\\
 &=\frac{2}{2\pi i}\int_{\delta-i\infty}^{\delta+i\infty}\left(\frac{2\pi}{\lambda} \right)^{s}\Phi(s)\frac{\Gamma(2s+2\rho+1)}{\Gamma(s+\rho+1)}\frac{1}{y^{2s+2\rho+1}}ds.
 \end{align*}
 Using the properties of $\Gamma$ function we have
 \begin{align*}
 U(y)=\frac{1}{2\pi i}\int_{(\delta)}\left(\frac{2\pi}{\lambda} \right)^{s}\Phi(s)\frac{\Gamma(s+\rho+\frac{1}{2})}{\sqrt{\pi}}\frac{2^{2s+2\rho+1}}{y^{2s+2\rho+1}}ds.
 \end{align*}
 Using \cite{BC} we evaluate $U(y)$ and have 
 \begin{align}\label{T2}
 U(y)&=\frac{1}{2\pi i}\int_{(2k-\delta)}\left(\frac{2\pi}{\lambda} \right)^{s}\Phi(s)\frac{\Gamma(s+\rho+\frac{1}{2})}{\sqrt{\pi}}\frac{2^{2s+2\rho+1}}{y^{2s+2\rho+1}}ds\nonumber \\
 &+\sum_{s\in polset\Phi(s)}\text{Res}\left(\frac{2\pi}{\lambda} \right)^{s}\Phi(s)\frac{\Gamma(s+\rho+\frac{1}{2})}{\sqrt{\pi}}\frac{2^{2s+2\rho+1}}{y^{2s+2\rho+1}}.
 \end{align} 
 Denote the first and second integrals as $T_{1}(y)$ and $T_{2}(y)$respectively so that  $U(y)=T_{1}(y)+T_{2}(y).$
 Using the identity, $\Phi(s)=i^{-2k}\left(\Phi(2k-s)-R(s)\right),$ and substituting in to the integrand of $T_{1}(y),$ we can write as 
 $$T_{1}(y) =V_{1}(y)-V_{2}(y),$$ where $V_{1}(y)$ and $V_{2}(y)$ are given as 
 \begin{align*}
 V_{1}(y)=\frac{1}{2\pi i}\int_{(2k-\delta)}\left(\frac{2\pi}{\lambda} \right)^{s} i^{-2k} \Phi(2k-s)\frac{\Gamma(s+\rho+\frac{1}{2})}{\sqrt{\pi}}\frac{2^{2s+2\rho+1}}{y^{2s+2\rho+1}}ds,
 \end{align*}
 \begin{align*}
 V_{2}(y)=\frac{1}{2\pi i}\int_{(2k-\delta)}\left(\frac{2\pi}{\lambda} \right)^{s} i^{-2k}R(s)\frac{\Gamma(s+\rho+\frac{1}{2})}{\sqrt{\pi}}\frac{2^{2s+2\rho+1}}{y^{2s+2\rho+1}}ds.
 \end{align*}
 Using the substitution $\vartheta =2k-s$ and $\Phi(\vartheta)=\left( \frac{2\pi}{\lambda}\right)^{-\vartheta}\Gamma(\vartheta)\varphi(\vartheta),$ where $\varphi(\vartheta)=\sum_{m=1}^{\infty}\frac{a_{m}}{m^{\vartheta}},$ replacing $\vartheta$ by $-\vartheta,$ we have  \begin{align*}
 V_{1}(y)=\frac{i^{-2k}}{\sqrt{\pi}} 2^{4k+2\rho+1}\Gamma\left(2k+\rho+\frac{1}{2} \right)\left(\frac{2\pi}{\lambda} \right)^{2k}\sum_{m=1}^{\infty} \frac{a_{m}}{\left( y^{2}+\left( \frac{4\pi}{\lambda}\right)^{2}m\right)^{2k+\rho+\frac{1}{2}}}.
 \end{align*} 
 The series converges for $y>0, $ provided  $2k+\rho-\frac{1}{2}> \beta.$ Since $\delta\geq \beta $ the series converges absolutely for $\rho\geq\delta-2k+\frac{1}{2}.$ \\
 Replacing the beta function by its equivalent $\Gamma$ function and replacing $R(s)$ by its equivalent representation we have
  \begin{align}\label{vtwo}
 V_{2}(y)=\frac{1}{\sqrt{2\pi}} \frac{2^{2\rho+1}}{y^{2\rho+1}}\sum_{j=1}^{p}\sum_{r=1}^{M_{j}}C_{rj}\left(\frac{-1}{\alpha_{j}} \right)^{r}\bigg\lbrace Q_{1}(y)-\alpha_{j}^{2k}Q_{2}(y)\bigg\rbrace,
 \end{align}
 where $$ Q_{1}(y)=\frac{1}{2\pi i}\int_{(2k-\delta)} \left(\frac{8\pi i\alpha_{j}}{\lambda y^{2}}\right)^{s}\frac{\Gamma(s)\Gamma(r-s)}{\Gamma(r)}\Gamma\left(s+\rho+\frac{1}{2}\right)ds$$ and
 $$ Q_{2}(y)=\frac{1}{2\pi i}\int_{(2k-\delta)} \left(\frac{8\pi}{i\lambda y^{2}\alpha_{j}}\right)^{s}\frac{\Gamma(2k-s)\Gamma(r+s-2k)}{\Gamma(r)}\Gamma\left(s+\rho+\frac{1}{2}\right)ds .$$ 
Using \cite{BC} we evaluate $Q_{1}(y)$ and obtain 
 \begin{align*}
 &\frac{1}{2\pi i}\int_{(2k-\delta)} \left(\frac{8\pi i\alpha_{j}}{\lambda y^{2}}\right)^{-s}\frac{\Gamma(-s)\Gamma(r+s)}{\Gamma(r)}\Gamma\left(\rho+\frac{1}{2}-s\right)ds\\
 &=\frac{1}{2\pi i}\int_{(N+\frac{1}{4})} \left(\frac{8\pi i\alpha_{j}}{\lambda y^{2}}\right)^{-s}\frac{\Gamma(-s)\Gamma(r+s)}{\Gamma(r)}\Gamma\left(\rho+\frac{1}{2}-s\right)ds\\
 &+\sum_{s\in Pole ~set~ of~ \Phi}\text{Res}\Bigg\lbrace  \left(\frac{8\pi i\alpha_{j}}{\lambda y^{2}}\right)^{-s}\frac{\Gamma(-s)\Gamma(r+s)}{\Gamma(r)}\Gamma\left(\rho+\frac{1}{2}-s\right)\Bigg\rbrace.
 \end{align*}
 It can be easily shown that the integral on the right-hand side tends to zero as $N$ tends to $\infty.$ Thus evaluating the residue of the poles we have 
 \begin{align*}
 Q_{1}(y)&=\sum_{[\delta]-2k+1}^{\infty} \left(\frac{8\pi i\alpha_{j}}{\lambda y^{2}} \right)^{-m}\frac{(-1)^{m}}{m!}\frac{\Gamma(m+r)}{\Gamma(r)}\Gamma\left(\rho-m+\frac{1}{2}\right)\\
 &+\sum_{m=0}^{\infty} \left(\frac{8\pi i\alpha_{j}}{\lambda y^{2}} \right)^{-m-\rho-\frac{1}{2}}\frac{(-1)^{m}}{m!}\frac{\Gamma(m+\rho+r+\frac{1}{2})}{\Gamma(r)}\Gamma\left(-\rho-m-\frac{1}{2}\right).
 \end{align*}
 Again using properties of the $\Gamma$ function and simplifying we obtain
 \begin{align*}
 Q_{1}(y)&=\sum_{m=0}^{\infty} \left(\frac{8\pi i\alpha_{j}}{\lambda y^{2}} \right)^{-m}\frac{(-1)^{m}}{m!}\frac{\Gamma(m+r)}{\Gamma(r)}(-1)^{m}\frac{\Gamma\left(\rho+\frac{1}{2} \right)\Gamma\left( -\rho+\frac{1}{2}\right)}{\Gamma\left(m-\rho+\frac{1}{2} \right)}\\
 & -\sum_{m=0}^{[\delta-2k]} \left(\frac{8\pi i\alpha_{j}}{\lambda y^{2}} \right)^{-m}\frac{(-1)^{m}}{m!}\frac{\Gamma(m+r)}{\Gamma(r)}\Gamma\left( \rho-m+\frac{1}{2}\right)\\
 &+\sum_{m=0}^{\infty} \left(\frac{8\pi i\alpha_{j}}{\lambda y^{2}} \right)^{-m-\rho-\frac{1}{2}}\frac{(-1)^{m}}{m!}(-1)^{m}\frac{\Gamma\left(-\rho-\frac{1}{2} \right)\Gamma\left( \rho+\frac{3}{2}\right)}{\Gamma\left(m+\rho+\frac{3}{2} \right)}\frac{\Gamma(m+\rho+r+\frac{1}{2})}{\Gamma(r)}.
 \end{align*}
 Using series representation of the confluent hypergeometric function we get
 \begin{align*}
 Q_{1}(y)&=\Gamma\left(\rho+\frac{1}{2} \right){_{1}F_{1}}\left( r,-\rho+\frac{1}{2} ; \frac{\lambda y^{2}}{8\pi i\alpha_{j}}\right)\\
 &+\frac{\Gamma\left(\rho+r+\frac{1}{2} \right)\Gamma\left(-\rho-\frac{1}{2} \right)}{\Gamma(r)}\left(\frac{\lambda y^{2}}{8\pi i\alpha_{j}} \right)^{\rho+\frac{1}{2}}{_{1}F_{1}}\left(\rho+r+\frac{1}{2},\rho+\frac{3}{2} ; \frac{\lambda y^{2}}{8\pi i\alpha_{j}}\right)\\
 &-\sum_{m=0}^{[\delta-2k]}\left( \frac{8\pi i\alpha_{j}}{y^{2}\lambda}\right)^{-m}\frac{(-1)^{m}}{m!}\frac{\Gamma(m+r)}{\Gamma(r)}\Gamma\left( \rho-m+\frac{1}{2}\right).
 \end{align*}
Using the the confluent hypergeometric function of the second kind we have  
 \begin{align}\label{Q1}
 Q_{1}(y)=&\Gamma\left(\rho+r+\frac{1}{2} \right)\Psi\left( r,-\rho+\frac{1}{2} ; \frac{\lambda y^{2}}{8\pi i\alpha_{j}} \right) \nonumber \\
 &-\sum_{m=0}^{[\delta]-2k}\left( \frac{8\pi i\alpha_{j}}{y^{2}\lambda}\right)^{-m}\frac{(-1)^{m}}{m!}\frac{\Gamma(m+r)}{\Gamma(r)}\Gamma\left( \rho-m+\frac{1}{2}\right).
 \end{align}
 \begin{align}\label{Q2}
 Q_{2}(y)&=\Gamma\left(2k+\rho+\frac{1}{2} \right)\left( \frac{i\alpha_{j} \lambda y^{2}}{8\pi}\right)^{r-2k} \Psi\left(r,r-\rho-2k+\frac{1}{2} ; \frac{i\alpha_{j} \lambda y^{2}}{8\pi}\right) \nonumber\\ 
 &- \sum_{m=0}^{[\delta]-r}\left( \frac{8\pi}{i\lambda y^2 \alpha_{j}}\right)^{2k-m-r}\frac{(-1)^{m}}{m!}\frac{\Gamma\left(m+r \right)}{\Gamma(r)}\Gamma\left(-m-r+\rho+\frac{1}{2}\right).
 \end{align}
 Thus for $\rho\geq 0, \rho\in\mathbb{Z}, \rho+2k \geq \delta+\frac{1}{2}, \delta>2k$ and $y\in\mathbb{R^{+}}$ the functional equation in $(\ref{func2})$ implies 
 \begin{align*}
 &U(y)=\frac{2^{2\rho+1}}{\sqrt{\pi}y^{2\rho+1}}a_{o}\Gamma(\rho+\frac{1}{2})+2^{\rho+1}\left(-\frac{1}{y}\frac{d}{dy} \right)^{\rho}\Bigg[\frac{1}{y}\sum_{m=1}^{\infty} a_{m}e^{-y\sqrt{m}}\Bigg]\\
 & =T_{1}(y)+T_{2}(y) = V_{1}(y)-V_{2}(y)+T_{2}(y)\\
 &=\frac{i^{-2k}}{\sqrt{\pi}} 2^{4k+2\rho+1}\Gamma\left(2k+\rho+\frac{1}{2} \right)\left(\frac{2\pi}{\lambda} \right)^{2k}\sum_{m=1}^{\infty} \frac{a_{m}}{\left( y^{2}+\left( \frac{4\pi}{\lambda}\right)^{2}m\right)^{2k+\rho+\frac{1}{2}}}-V_{2}(y)+T_{2}(y).
 \end{align*}
 Equating the left-hand and right-hand side expression we obtain
 \begin{align}\label{T21}
 & 2^{\rho+1}\left(-\frac{1}{y}\frac{d}{dy} \right)^{\rho}\Bigg[\frac{1}{y}\sum_{m=1}^{\infty} a_{m}e^{-y\sqrt{m}}\Bigg]=\frac{i^{-2k}}{\sqrt{\pi}} 2^{4k+2\rho+1}\Gamma\left(2k+\rho+\frac{1}{2} \right)\left(\frac{2\pi}{\lambda} \right)^{2k}\sum_{m=1}^{\infty} \frac{a_{m}}{\left( y^{2}+\left( \frac{4\pi}{\lambda}\right)^{2}m\right)^{2k+\rho+\frac{1}{2}}}\nonumber\\
 &-V_{2}(y)+T_{2}(y)-\frac{2^{2\rho+1}}{\sqrt{\pi}y^{2\rho+1}}a_{o}\Gamma(\rho+\frac{1}{2}),
 \end{align}
 where $T_{2}(y)$ and $V_{2}(y)$ are given by $(\ref{T2})$ and $(\ref{vtwo})$ respectively. We can also get an explicit formula for $V_{2}(y)$ from $(\ref{T2})$ and $(\ref{Q1}).$  \\
 
 Using the expression for $T_2$ from (\ref{T2}) and the residue of $\Phi(s)$ from theorem 2.1 we have
 \begin{align}\label{T22}
 T_{2}(y)&=\frac{2^{2\rho+1}}{\sqrt{\pi} y^{2\rho+1}}\Bigg\lbrace\left( \frac{8\pi i}{\lambda y^2}\right)^{2k}\Gamma\left(2k+\rho+\frac{1}{2} \right)-a_{0}\Gamma\left(\rho+\frac{1}{2} \right)\nonumber\\
 &+\sum_{m=k}^{L}C_{m} \bigg\lbrace \left( \frac{8\pi i}{\lambda y^2}\right)^{2k-m}\Gamma\left(2k-m+\rho+\frac{1}{2} \right)-\left( -\frac{8\pi i}{\lambda y^2}\right)^{m}\Gamma\left(m+\rho+\frac{1}{2} \right)  \bigg\rbrace \nonumber \\
 &+\sum_{j=1}^{p}\sum_{r=1}^{M_{j}}C_{rj}\Bigg[ \sum_{m=0}^{[\delta-r]} \frac{\Gamma(m+r)}{\Gamma(r)}\frac{(-1)^{m}}{m!}\left( \frac{8\pi}{\lambda y^2}\right)^{2k-m-r}i^{m+r-2k} \alpha_{j}^{m}\times\nonumber\\
 &~~~~~~~~\Gamma\left(2k-m-r+\rho+\frac{1}{2} \right) \nonumber\\ 
 &-\sum_{m=0}^{[\delta-2k]} \frac{\Gamma(m+r)}{\Gamma(r)}\frac{(-1)^{m}}{m!}\left( \frac{\lambda y^2 }{8\pi}\right)^{m}i^{m}\Gamma\left(-m+\rho+\frac{1}{2} \right)\alpha_{j}^{-m-r}\Bigg] \Bigg\rbrace.
 \end{align}
 Now substituting(\ref{T22}) into (\ref{T21})  for $T_{2}(y)$ and by substituting  the respective expressions $(\ref{Q1}),   (\ref{Q2})$  into $V_{2}(y),$  we obtain the identity (\ref{iden2}) and this completes the proof of implication part.\\
 
 To prove the converse we only show $(\ref{iden2}) $ implies   $(\ref{func2}).$ 
  Multiply (\ref{thmmI}) by $e^{y\sqrt{x}},$ with Re$y>0$ and $x>0$ and integrate the expression along vertical path Re$s= \vartheta,$ where $\vartheta>0.$  The left hand side (\ref{iden2}) can be evaluated using the formula (\cite{CN}, page 9) 
 \begin{align}\label{conver2}
 \sum_{m=1}^{\infty}a_{m}\frac{1}{2\pi i} \int_{(\vartheta)}e^{y\sqrt{x}}\left(-\frac{1}{y}\frac{d}{dy} \right)^{\rho}\bigg[\frac{1}{y}e^{-y\sqrt{m}}\bigg]dy=\frac{1}{\Gamma(\rho+1)}{\sum_{m\leq x}}^{\prime} a_{m}(x-m)^{\rho}2^{-\rho},
 \end{align} 
 For the right-hand side of (\ref{iden2}) we compute the integral of each term one by one. So put the first term
 \begin{align*}
 P(x)=\frac{i^{-2k}}{\sqrt{\pi}}2^{4k+\rho} \Gamma\left( 2k+\rho+\frac{1}{2}\right)\left(\frac{2\pi}{\lambda} \right)^{2k}\sum_{m=1}^{\infty}a_{m}\frac{1}{2\pi i}\int_{(\vartheta)} \frac{e^{y\sqrt{x}}}{[y^2+\left( \frac{4\pi}{\lambda}\right)^{2}m ]^{2k+\rho+\frac{1}{2}}}dy.
 \end{align*}
 Using Sterling formula and simplifying we obtain 
 
  \begin{align}\label{exp1}
 P(x)=i^{-2k}\left(\frac{4\pi}{\lambda} \right)^{-\rho} \sum_{m=1}^{\infty} a_{m}\left(\frac{x}{m} \right)^{\frac{2k+\rho}{2}}J_{2k+\rho}\left(\frac{4\pi\sqrt{mx}}{\lambda} \right).
 \end{align}
 Put the second term
 \begin{align*}
 I_{1}(x)=\frac{1}{2\pi i}\int_{(\vartheta)}\frac{e^{y\sqrt{x}}}{y^{2\rho+1}}\Psi\left(r,-\rho+\frac{1}{2};\frac{\lambda y^{2}}{8\pi i\alpha_{j}} \right)dy.
 \end{align*} 
 Using Sterling formula we have 
 \begin{align*}
 I_{1}(x)&=\frac{1}{2\pi i}\int_{(\vartheta)}\frac{e^{y\sqrt{x}}}{y^{2\rho+1}} \bigg\lbrace\frac{1}{2\pi i}\int_{(\theta)}\frac{\Gamma(r+\tau)\Gamma(-\tau)\Gamma(\rho+\frac{1}{2}-\tau)}{\Gamma(r)\Gamma(r+\rho+\frac{1}{2})}\left(\frac{\lambda y^{2}}{8\pi i\alpha_{j}} \right)^{\tau}d \tau\bigg\rbrace dy\\
 &=\frac{1}{2\pi i}\int_{(\theta)} \Bigg\lbrace\frac{\Gamma(r+\tau)\Gamma(-\tau)\Gamma(\rho+\frac{1}{2}-\tau)}{\Gamma(r)\Gamma(r+\rho+\frac{1}{2})}\left(\frac{\lambda }{8\pi i\alpha_{j}} \right)^{\tau}\frac{1}{2\pi i}\int_{(\vartheta)}e^{y\sqrt{x}}y^{2\tau-2\rho-1}dy\Bigg\rbrace d \tau\\
 &=\frac{1}{2\pi i}\int_{(\theta)} \frac{\Gamma(r+\tau)\Gamma(-\tau)\Gamma(\rho+\frac{1}{2}-\tau)}{\Gamma(r)\Gamma(r+\rho+\frac{1}{2})}\left(\frac{\lambda }{8\pi i\alpha_{j}} \right)^{\tau}\frac{x^{\rho-\tau}}{\Gamma(2\rho-2\tau+1)}d\tau.
 \end{align*}
Using properties of the $\Gamma$ function we obtain
 \begin{align*}
 I_{1}(x)=\frac{\sqrt{\pi}x^{\rho}}{2^{2\rho}\Gamma(r+\frac{1}{2}+\rho)\Gamma(r)}\Bigg[\frac{1}{2\pi i}\int_{(\theta)}\frac{\Gamma(r+\tau)\Gamma(-\tau)}{\Gamma(\rho+1-\tau)} \left(\frac{\lambda }{2\pi i\alpha_{j}} \right)^{\tau}d\tau\Bigg].
 \end{align*} 
Using \cite{BC} we obtain 
 \begin{align*}
 I_{1}(x)=\frac{\sqrt{\pi}x^{\rho}}{2^{2\rho}\Gamma(r+\frac{1}{2}+\rho)\Gamma(r)}\sum_{m=0}^{\infty}\frac{(-1)^{m}}{m!} \frac{\Gamma(m+r)}{\Gamma(\rho+m+r+1)}\left(\frac{2\pi i\alpha_{j}x}{\lambda} \right)^{m+r}.
 \end{align*}
 Therefore, we conclude that
 \begin{align}\label{exp2}
 \frac{1}{2\pi i}\int_{(\vartheta)}\frac{e^{y\sqrt{x}}}{\sqrt{\pi}y^{2\rho+1}}2^{\rho}\sum_{j=1}^{p}\sum_{r=1}^{M_{j}}C_{ri}\left( \frac{-1}{\alpha_{j}}\right)^{r}\Gamma(\rho+r+\frac{1}{2})\Psi\left(r,-\rho+\frac{1}{2};\frac{\lambda y^{2}}{8\pi i\alpha_{j}} \right)dy \nonumber  \\
 =\frac{1}{2^{\rho}}\sum_{j=1}^{p}\sum_{r=1}^{M_{j}}C_{rj}\left( \frac{-1}{\alpha_{j}}\right)^{r}\left(\frac{2\pi i\alpha_{j}}{\lambda} \right)^{r}\frac{x^{\rho+r}}{\Gamma(\rho+r+1)}{_{1}F_{1}}\left(r,r+\rho+1;\frac{-2\pi i\alpha_{j}x}{\lambda} \right).
 \end{align}
  Next put 
 \begin{align*}
 I_{2}(x)=\frac{1}{2\pi i}\int_{(\vartheta)}\frac{e^{y\sqrt{x}}}{y^{2\rho+1}}\left(\frac{\alpha_{j}i\lambda y^{2}}{8\pi} \right)^{r-2k} \Psi\left(r,r-2k-\rho+\frac{1}{2};\frac{i\alpha_{j}\lambda ^{2}}{8\pi} \right) dy. 
 \end{align*}
 For fixed $r, -r<\theta<-r+\frac{1}{2}$ and $-\frac{\pi}{2}\leq arg\left(\frac{i\alpha_{j}\lambda ^{2}}{8\pi} \right),$  using the integral representation of confluent hypergeometric function of the second kind  we get 
 \begin{align*}
 I_{2}(x)=\frac{1}{2\pi i}\int_{(\vartheta)}\frac{e^{y\sqrt{x}}}{y^{2\rho+1}}\left(\frac{\alpha_{j}i\lambda y^{2}}{8\pi} \right)^{r-2k}\Bigg\lbrace \frac{1}{2\pi i}\int_{(\theta)} \frac{\Gamma(r+\tau)\Gamma(-\tau)\Gamma\left(2k+\rho-\tau-r+\frac{1}{2} \right)}{\Gamma(r)\Gamma(2k+\frac{1}{2}+\rho)} \\
 \times\left(\frac{\alpha_{j}i\lambda y^{2}}{8\pi} \right)^{\tau}d\tau \Bigg\rbrace dy.
 \end{align*}
 By interchanging the order of integration which can be justified we obtain  
 \begin{align*}
 I_{2}(x)=\frac{1}{2\pi i}\int_{(\theta)}\Bigg[ \frac{\Gamma(r+\tau)\Gamma(-\tau)\Gamma\left(2k+\rho-\tau-r+\frac{1}{2} \right)}{\Gamma(r)\Gamma(2k+\frac{1}{2}+\rho)}\left(\frac{\alpha_{j}i\lambda}{8\pi} \right)^{\tau+r-2k}\\
 \times \frac{1}{2\pi i}\int_{(\vartheta)}e^{y\sqrt{x}} y^{2\tau+2r-4k-2\rho-1}dy\Bigg]  d\tau. 
 \end{align*}
Computing the last integral and simplifying we obtain 
 \begin{align*}
 I_{2}(x)=\frac{1}{2\pi i}\int_{(\theta)} \Bigg[\frac{\Gamma(r+\tau)\Gamma(-\tau)\Gamma\left(2k+\rho-\tau-r+\frac{1}{2} \right)}{\Gamma(r)\Gamma(2k+\frac{1}{2}+\rho)}\left(\frac{\alpha_{j}i\lambda}{8\pi} \right)^{\tau+r-2k}\\
 \times \frac{x^{2k+\rho-r-\tau}}{\Gamma(4k+2\rho+1-2r-2\tau)}\Bigg]d\tau.
 \end{align*}
 Using properties of the $\Gamma$ function we obtain
$$\displaystyle I_{2}(x) =\frac{\sqrt{\pi}}{\Gamma(r)\Gamma\left( 2k+\rho+\frac{1}{2}\right)}\left( \frac{i \alpha_{j} \lambda}{8\pi}\right)^{r-2k} \left(\frac{x}{4}\right)^{2k+\rho-r}\times\\ 
 \frac{1}{2\pi i} \int_{(\theta)}\frac{\Gamma(r+\tau)\Gamma(-\tau)}{\Gamma(2k+\rho+1-r-\tau)}\left( \frac{i\alpha_{j}\lambda}{2\pi x}\right)^{\tau}d\tau.$$\\
Using \cite{BC}) we obtain
 $$I_{2}(x)=\frac{\sqrt{\pi}}{\Gamma(r)\Gamma\left( 2k+\rho+\frac{1}{2}\right)}\left( \frac{i \alpha_{j} \lambda}{8\pi}\right)^{-2k} \frac{x^{2k+\rho}}{2^{4k+2\rho-2r}}  \sum_{m=0}^{\infty}\frac{(-1)^{m}}{m!}\frac{\Gamma(m+r)}{\Gamma(2k+\rho+m+1)}\left(\frac{2\pi x}{i\alpha_{j}\lambda} \right)^{m}.$$ 
  Thus using this result for $I_{2}(x)$  we find that 
 \begin{align}\label{exp3}
 \frac{1}{2\pi i}\int_{(\vartheta)}\frac{2^{\rho} e^{y\sqrt{x}}}{\sqrt{\pi}y^{2\rho+1}}\sum_{j=1}^{p}\sum_{r=1}^{M_{j}}C_{ri}\left( \frac{-1}{\alpha_{j}}\right)^{r} \alpha_{j}^{2k}\Gamma(2k+\rho+\frac{1}{2}) \left(\frac{\alpha_{j} i\lambda y^{2}}{8\pi} \right)^{r-2k} \nonumber \\
 \times\Psi\left(r,r-2k-\rho+\frac{1}{2};\frac{i\alpha_{j} \lambda^{2}}{8\pi } \right)dy \nonumber \\
 =\frac{1}{2^{\rho}}\sum_{j=1}^{p}\sum_{r=1}^{M_{j}}C_{ri}\left( \frac{-1}{\alpha_{j}}\right)^{r}\left( \frac{2\pi}{i\lambda}\right)^{2k}\frac{x^{2k+\rho}}{\Gamma(2k+\rho+1)}{_{1}F_{1}}\left( r,2k+\rho+1 ;\frac{-2\pi x}{i \alpha_{j}\lambda}\right).
 \end{align}
 Put the third term as $I_3(x)$ which is given by 
 \begin{align*}
 I_{3}(x)=\sum_{m=k}^{L} C_{m}\frac{2^{\rho}}{\sqrt{\pi}}\Bigg\lbrace\frac{1}{2\pi i}\int_{(\vartheta)} \frac{e^{y\sqrt{x}}}{y^{2\rho+1}}\Bigg[\left( \frac{8\pi i}{\lambda y^{2}}\right)^{2k-m} \Gamma\left(2k-m+\rho+\frac{1}{2} \right)\\
 -\left( \frac{-8\pi i}{\lambda y^{2}}\right)^{m} \Gamma\left(m+\rho+\frac{1}{2} \right)\Bigg]dy\Bigg\rbrace.
 \end{align*}
 Evaluating each  integrals in ~$I_{3}(x)$~ and simplifying, we get
 \begin{align}\label{exp4}
 I_{3}(x)&=2^{\rho}\sum_{m=k}^{L}C_{m}\Bigg\lbrace \left( \frac{8\pi i}{\lambda }\right)^{2k-m}\frac{x^{k+m}}{2^{4k+2\rho-2m}\Gamma(2k+\rho-m+1)}\nonumber \\
 &-\left( \frac{-8\pi i}{\lambda }\right)^{m}\frac{x^{2k+\rho-m}}{2^{2k+2m}\Gamma(\rho+m+1)}\Bigg\rbrace \nonumber \\
 &=\frac{1}{2^{\rho}}\sum_{m=k}^{L}C_{m}\Bigg\lbrace \left( \frac{2\pi}{\lambda }\right)^{2k-m}\frac{i^{2k-m} x^{2k+\rho-m}}{\Gamma(2k+\rho-m+1)}-\left( \frac{2\pi}{\lambda }\right)^{m}\frac{ (-i)^{m} x^{\rho+m}}{\Gamma(\rho+m+1)}\Bigg\rbrace.
 \end{align}
 Finally we evaluate the last term in (\ref{iden2}) which we denote by $I_4(x)$ as
 \begin{align}\label{exp5}
 I_{4}(x)=-\frac{2^{\rho+1}}{\sqrt{\pi}} a_{0}\Gamma\left( \rho+\frac{1}{2}\right)\frac{1}{2\pi i}\int_{(\vartheta)}\frac{e^{y\sqrt{x}}}{y^{2\rho+1}}dy=-\frac{a_{0}x^{\rho}}{2^{\rho-1}\Gamma(\rho+1)}. 
 \end{align}
  Thus  combining the results in (\ref{exp1}), (\ref{exp2}), (\ref{exp3}), (\ref{exp4}), and  (\ref{exp5}) we obtain 
 \begin{align*}
 &\frac{1}{\Gamma(\rho+1)}{\sum_{0\leq m\leq x}}^{\prime} a_{m}(x-m)^{\rho}
 =i^{-2k}\left(\frac{2\pi}{\lambda} \right)^{-\rho} \sum_{m=1}^{\infty} a_{m}\left(\frac{x}{m} \right)^{\frac{2k+\rho}{2}}J_{2k+\rho}\left(\frac{4\pi\sqrt{mx}}{\lambda} \right)\\
 & +i^{2k}\left(\frac{2\pi}{\lambda} \right)^{2k} \frac{x^{2k+\rho}a_{0}}{\Gamma(2k+\rho+1)} \\
 &+\sum_{m=k}^{L}C_{m}\Bigg\lbrace \left( \frac{2\pi}{\lambda }\right)^{2k-m}\frac{i^{2k-m} x^{2k+\rho-m}}{\Gamma(2k+\rho-m+1)}-\left( \frac{2\pi}{\lambda }\right)^{m}\frac{ (-i)^{m} x^{\rho+m}}{\Gamma(\rho+m+1)}\Bigg\rbrace\\
 &-\sum_{j=1}^{p}\sum_{r=1}^{M_{j}}C_{rj}\left( \frac{-1}{\alpha_{j}}\right)^{r}\left(\frac{2\pi i\alpha_{j}}{\lambda} \right)^{r}\frac{x^{\rho+r}}{\Gamma(\rho+r+1)}{_{1}F_{1}}\left(r,r+\rho+1;\frac{-2\pi i\alpha_{j}x}{\lambda} \right)\\
 &+\sum_{j=1}^{p}\sum_{r=1}^{M_{j}}C_{ri}\left( \frac{-1}{\alpha_{j}}\right)^{r}\left( \frac{2\pi}{i\lambda}\right)^{2k}\frac{x^{2k+\rho}}{\Gamma(2k+\rho+1)}{_{1}F_{1}}\left( r,2k+\rho+1 ;\frac{-2\pi x}{i \alpha_{j}\lambda}\right).
 \end{align*}
  This completes the proof of the converse and therefore, the proof of theorem is  completed.
 \end{proof}
 
 \end{document}